\documentclass[11pt]{amsart}
\usepackage[mathscr]{eucal}
\usepackage{graphicx}
\usepackage{float}
\restylefloat{figure}

\usepackage{fullpage}
\newtheorem{theorem}{Theorem}[section]
\newtheorem{lemma}[theorem]{Lemma}

\newtheorem{assumption}[theorem]{Assumption}
\newcommand\be{\begin{enumerate}}
\newcommand\ee{\end{enumerate}}

\numberwithin{equation}{section}

\newtheorem{example}[theorem]{Example}

\newtheorem{remark}[theorem]{Remark}

\begin{document}

\title[Proportional feedback control with noise]{Stabilization of difference equations 
with noisy proportional feedback control}
\author[E. Braverman and A. Rodkina]
{E. Braverman and A. Rodkina}

\address{Department of Mathematics, University of Calgary, Calgary, Alberta T2N1N4, Canada}
\email{maelena@math.ucalgary.ca}
\address{Department of Mathematics  \\
The University of the West Indies, Mona Campus, Kingston, Jamaica}
\email{alexandra.rodkina@uwimona.edu.jm}

\thanks{The authors were supported by the NSERC grant RGPIN-2015-05976 and American Institute of Mathematics SQuaRE program}

\begin{abstract}
Given  a deterministic difference equation $x_{n+1}= f(x_n)$, we would like 
to stabilize any point $x^{\ast}\in (0, f(b))$, where $b$ is a unique maximum point of $f$,  
by introducing  proportional feedback (PF) control. We assume that  PF control contains  
either a multiplicative $x_{n+1}= f\left( (\nu + \ell\chi_{n+1})x_n \right)$ 
or an additive noise $x_{n+1}=f(\lambda x_n) +\ell\chi_{n+1}$.  We study conditions under which the 
solution eventually enters some interval, treated as a stochastic (blurred) equilibrium. 
In addition, we prove that, for each $\varepsilon>0$,  when the noise level  $\ell$ is sufficiently small, 
all solutions eventually belong to the interval $(x^{\ast}-\varepsilon,x^{\ast}+\varepsilon)$.


\end{abstract}

\maketitle

\section{Introduction}
\label{sec:intr}

Difference equations can describe population models with non-overlapping generations. 
For simplicity, the parameters involved are combined and reduced to a single value 
in a one-dimensional model.  
In the logistic equation this parameter describes an intrinsic growth rate.
While demonstrating stable dynamics for small parameter values, as the parameter grows, solutions can become 
chaotic. Thus, even for simplest one-dimensional models,  solution behaviour can become complex, resulting in the infinitesimal difference between the two initial values being amplified, leading to unpredictable dynamics. To avoid this situation, either stepwise or periodic types of control were suggested.  
One of efficient methods of chaos control is the proportional feedback (PF) which consists in the reduction of the state variable at each step (or selected steps), corresponding to harvesting or pest management. For a one-dimensional map
\begin{equation}
\label{1}
x_{n+1}=f(x_n),\quad x_0>0, \quad n\in \mathbb N,
\end{equation}
PF method involves reduction  of the state variable at each step,  proportional to the size of the variable
\begin{equation}
\label{2}
x_{n+1}=f(\nu x_n),\quad x_0>0 , \quad n\in \mathbb N, \quad \nu \in (0,1].
\end{equation}
This type of control may describe harvesting with a constant effort  \cite{clark}, where the harvest 
is proportional to the population size, or pest control with constant efficiency \cite{seno,zipkin}. Sometimes reduction boosts population sizes 
\cite{seno,zipkin}, this is called the hydra effect, see, for example, \cite{LH2012} and its literature list.
Stabilization with proportional feedback control was recently studied in \cite{NODY,Carmona,Liz2010}, 
see also references therein.

Since in many cases there is stochasticity involved in the control, 
it is reasonable to model population dynamics processes with stochastic equations.  Stability of 
solutions of stochastic difference equations was considered in several publications, see  
\cite{ABR,AKMR,AMR,Shai,BR5,KR09,Shai} and references therein. Paper \cite{ABR} proves 
non-exponential stability and estimates decay rates for solutions of nonlinear equations with unbounded noise, 
\cite{AKMR} and \cite {KR09} are focused on  local dynamics of polynomial equations with fading 
stochastic perturbations, 
in \cite {AMR} stabilization of difference equations by noise was introduced and justified, 
\cite{Shai} is devoted to optimal control for Volterra equations. 
In \cite{BKR6, BR2,  BR3, BR5} stochastic equations connected with population models were 
considered: \cite{BR2,  BR3} discuss stabilization of two-cycles of equations under stochastic perturbations,  \cite{BR5} deals with convergence of solutions of equations with additive perturbations, 
in \cite{BKR6} stabilization of difference equations with noisy prediction-based control was presented.

Let us assume that the harvesting effort (or management efficiency) varies resulting in a  multiplicative noise
\begin{equation}
\label{3}
x_{n+1}= f\left( (\nu + \ell\chi_{n+1})x_n \right),\quad x_0>0, \quad n\in \mathbb N. 
\end{equation}
It is also possible to suppose that the  management efficiency is constant but the environment is stochastic, 
leading to an additive noise
\begin{equation}
\label{4}
x_{n+1}= \max \left\{ f(\nu x_n) +\ell\chi_{n+1}, 0 \right\}.
\end{equation}
In this paper we impose the  assumption on the function $f$ which describes its 
behaviour in a right neighbourhood  of zero. 

\begin{assumption}
\label{as:slope}
The function $f : [0, \infty) \rightarrow [0, \infty)$ is continuous, and
there is a real number $b>0$ such
that $f(x)$ is strictly monotone increasing, while the function $\frac{f(x)}{x}$ is strictly monotone
decreasing on $(0,b]$, $f(b)>b$, while $\frac{f(b)}{b} > \frac{f(x)}{x}$ for any $x \in (b, \infty)$.
\end{assumption}

\begin{remark}
\label{rem:bb1}
 Note that, once  Assumption \ref{as:slope} holds for a certain $b>0$, it is also satisfied for any $b_1\in (0, b]$.
\end{remark}
For example, the Ricker model
\begin{equation}
\label{eq:ricker}
x_{n+1} = f_1(x_n) = x_n e^{r(1-x_n)}
\end{equation}
satisfies Assumption~\ref{as:slope} with $b \leq 1/r$,
the truncated logistic model 
\begin{equation}
\label{eq:logistic}
x_{n+1} = f_2(x_n) = \max\left\{ rx_n (1-x_n), 0 \right\}
\end{equation}
with $r>2$ and $b \leq \frac{1}{2}$.  The modifications of the Beverton-Holt equation
$$
x_{n+1} = f_3(x_n) = \frac{Ax_n}{1+Bx_n^{\gamma}},\quad A,B>0,~~\gamma>1, 
$$
which is also called a Maynard Smith model \cite{Thieme},
and
$$
x_{n+1} = f_4(x_n) = \frac{Ax_n}{(1+Bx_n)^{\gamma}},\quad A,B>0,~~\gamma>1
$$
satisfy Assumption~\ref{as:slope}, with the parameters which lead to the existence of a positive equilibrium exceeding the  unique maximum point $x_{\max}$  and $b<x_{\max}$. In all the above examples, $f_i$ are unimodal: they increase for $x \in [0,x_{\max}]$ and decrease on $[x_{\max},\infty)$, with the only critical point on $[0,\infty)$ being a global maximum.

Note that in  Assumption~\ref{as:slope} we do not assume $f(0)=0$, though in most practical examples, this condition is satisfied.  
However, $f(0)=0$ does not necessarily imply that a finite limit $\displaystyle \lim_{x\to 0^+}\frac{f(x)}{x}$ exists: the Gompertz model with 
$\displaystyle f(x)=x \ln \left(\frac{K}{x} \right)$ satisfies Assumption~\ref{as:slope} with $b \leq K/e$, while this 
limit is $+\infty$.   
The function $f$ in \eqref{1} does not need to  be unimodal in order to satisfy  Assumption~\ref{as:slope}.
For example, the equation describing the growth of bobwhite quail populations \cite{Milton} 
\begin{equation} 
\label{eq:milton}  
x_{n+1} = x_n \left( A+ \frac{B}{1+x_n^{\gamma}} \right),\quad A,B>0,~~\gamma>1,
\end{equation}
in the right-hand side involves a function, which, for some $A$ and $B$, is bimodal, 
i.e. has two critical points, the smaller of them is a local maximum,  while the larger one is a local minimum.  
However, Assumption~\ref{as:slope} is satisfied for the function $f(x_n)$ in the right-hand side of \eqref{eq:milton}, with $b$ not exceeding the  smallest  positive critical point.

In \cite{NODY}, stabilization of any point $x^{\ast} \in (0,b)$ using \eqref{2} was considered  
under the following conditions: $f$ increases on $(0,b)$ and satisfies $f(x)>x$,   
$f'(x)>0$ and $f''(x)<0$ for $x \in (0,b)$, as well as 
${\displaystyle \frac{f(x)}{x} < \frac{f(b)}{b} \, ,\quad\forall\;x>b}$. 

These conditions imply asymptotic stability \cite{NODY} and are more
restrictive than Assumption~\ref{as:slope}, which, with the appropriate control, also ensures asymptotic stability 
\cite{Brian} of an arbitrarily chosen point $x^{\ast} \in (0,b)$, for an appropriate $\nu$. Thus, in the present paper we stick up to less restrictive Assumption~\ref{as:slope}. The condition $f(x)>x$, $x\in (0,b]$ means that all positive 
fixed points $x$ of the original function  $f$ satisfy $x > b$. 

However, in the stochastic case, where the noise does not tend to zero as time grows,
stabilization in regular sense is impossible: at any stage, a solution will deviate from the equilibrium.
We study conditions under which the solution eventually enters some interval,
which can be treated as a stochastic (blurred) equilibrium.  
More exactly, in this paper we stochastically stabilize each point $x^*\in (f(0), f(b))$. 
First,  for  each  $x^*\in (f(0), f(b))$ we find $\nu=\nu(x^*)\in (0, 1)$ such that $x^*$ becomes  
a fixed  point for the function $g(x):=f(\nu x)$.  
Afterwards, we determine the maximum noise intensity  $\ell$ for which we can guarantee stability 
of the equilibrium $x^*$ in some stochastic sense.  

The approach and the results are  different in the two cases of equations with multiplicative noise \eqref{3} 
and with additive noise \eqref{4}. 
The main difference is that for the multiplicative noise our results hold almost surely, 
while for the additive noise we can prove them only with some probability. 
More exactly, for the solution $x_n$ of equation \eqref{3} with a multiplicative noise  
and an arbitrary positive  initial value $x_0$ we show, for each $\varepsilon>0$, that 
$x_n\in (y_1-\varepsilon, y_2 +\varepsilon)$ almost surely, for $n\ge N$, where $N$ is large enough. 
Here the points $y_1$ and $y_2$ satisfying $0<y_1<y_2$  and nonrandom $N$ are 
defined by the point $x^*$, which we are stabilizing, and by the noise intensity $\ell$. We also 
show that when $\ell\to 0$, the above interval is shrinking to $(x^*-\varepsilon, 
x^*+\varepsilon)$.
For the solution $x_n$ of equation \eqref{4}  with an additive noise and an arbitrary positive 
initial value $x_0$, we show, for each $\varepsilon>0$  and $\gamma \in (0, 1)$, the existence of a 
nonrandom $N=N(\varepsilon, x_0, \gamma)$ such that  probability $\displaystyle \mathbb P \{y_1\le x_n\le  
y_2+\varepsilon, \,\, \text{for} \,\,  n\ge N\}>\gamma$. In fact, the closer $\gamma$ is to 1, the bigger
becomes $N=N(\varepsilon, x_0, \gamma)$.  We also show that when 
$\ell\to 0$, with probability $\gamma$,  the above interval tends to $(x^*-\varepsilon, x^*+\varepsilon)$.


The paper is organized as follows.  In Section~\ref{sec:prelim} we introduce definitions, assumptions and prove some  preliminary results. Section \ref{sec:multi} considers multiplicative noise. 
In order to obtain in Section~\ref{subsec:stochmulti} the main result for the stochastic equation, 
we first prove several auxiliary statements for  a deterministic 
equation with variable PF control in Section~\ref{subsec:detmulti}. 
Section~\ref{sec:add} deals with additive stochastic perturbations. Similarly to the multiplicative case, first,  in Section~\ref{subsec:detadd}
variable deterministic perturbations are studied, and then the main stochastic result is obtained in Section \ref{subsec:stochadd}.  Section \ref{sec:examples} presents numerical examples illustrating the possibility of stabilization, while Section \ref{sec:discussion} discusses further developments and new problems arising from the present  research.

\section{Definitions  and  preliminaries}
\label{sec:prelim}

Let  $(\Omega, {\mathcal{F}},  (\mathcal{F}_n)_{n \in \mathbb{N}}, {\mathbb{P}})$ be  a complete 
filtered probability space, $\chi:=(\chi_n)_{n\in\mathbb{N}}$  be a sequence of independent random 
variables with the zero mean. The filtration $(\mathcal{F}_n)_{n \in \mathbb{N}}$ is supposed to be naturally generated by  the sequence $(\chi_n)_{n\in\mathbb{N}}$, i.e.
$\mathcal{F}_{n} = \sigma \left\{\chi_{1},  \dots, \chi_{n}\right\}$.

In the present paper, we consider \eqref{3} and \eqref{4},  
where the sequence $(\chi_n)_{n \in {\mathbb N}}$ satisfies the following condition.
\begin{assumption}
\label{as:chibound}
$(\chi_n)_{n \in {\mathbb N}}$ is a sequence of independent and identically  distributed  
continuous random variables, with the density function $\phi(x)$ such that
\[
\phi(x)>0, \quad x\in [-1, 1], \quad \phi(x)\equiv 0, \quad x\notin [-1, 1].
\]
\end{assumption}

We use the standard abbreviation ``a.s." for the wordings ``almost sure" or ``almost surely" with respect
to the fixed probability measure $\mathbb P$  throughout the text. 
A detailed discussion of stochastic concepts and notation can be found, for example, in \cite{Shiryaev96}.

Now we state some auxiliary results concerning the function $f(\nu x)$, for any $\nu \in (0,1]$.
Let Assumption \ref{as:slope} hold, then the function $f$ is  increasing and continuous 
when  restricted to the interval $(0, b)$. Thus $f$ has the  increasing  and continuous inverse  function 
\[
f^{-1}:[f(0), f(b)]\to [0, b].
\] 
Define
\begin{equation}
\label{def:Phi}
\Phi(z):=\frac z{f(z)}, \quad z\in (0, b). 
\end{equation}

By Assumption \ref{as:slope} the  function $\Phi$ is increasing, continuous and therefore uniquely 
invertible on $(0, b)$. The domain of the inverse function depends on the value of $f(0)$ at zero.
Under Assumption \ref{as:slope}, the limit 
\begin{equation*}
\label{def:lim}
 \lim_{x \to 0^+} \frac{f(x)}{x}
\end{equation*}
exists (finite or infinite), is positive and greater than 1. In fact, the  function $f(x)/x$ is 
decreasing on $(0,b)$ and $f(b)/b > 1$.  Note that if $f(0)>0$ then $\lim_{x \to 0^+} \frac{f(x)}{x}=+\infty$, however, 
in the case $f(0)=0$ the value of the limit   can be finite, as well as infinite.   Since
\[
\frac{f(x)}{x}=\frac 1{\Phi(x)},
\]
we have
\[
\lim_{x \to 0^+} \frac{f(x)}{x}=\lim_{x \to 0^+} \frac 1{\Phi(x)}.
\]
Therefore,  we can set
\begin{equation}
\label{def:phi0}
 \Phi(0): =\left\{\begin{array}{cc}
  0,  \hspace{-0.9cm} & \hspace{-0.9cm}\text{if} \quad \lim\limits_{x \to 0^+} \frac{f(x)}{x}=\infty;  \\ \\
  \lim\limits_{x \to 0^+} \frac {x}{f(x)}, \quad &\text{if $\lim\limits_{x \to 0^+} \frac{f(x)}{x}\in (1, \infty)$ } . 
 \end{array}\right.
\end{equation}
Thus, we have proved the following lemma.
\begin{lemma}
\label{lem:phi-1}
Let Assumption \ref{as:slope} hold, and $\Phi$ be defined as in \eqref{def:Phi} and \eqref{def:phi0}.  Then
\be
\item [(i)] $\Phi: (0, b)\to \left( \Phi(0),  \Phi(b)\right),\quad \Phi^{-1}: \biggl(\Phi(0),  \Phi(b)\biggr) \to (0, b)$;\\
\item [(ii)] $0\le \Phi(0)<\Phi(b)<1$;\\
\item [(iii)]  both functions, $\Phi$ and $\Phi^{-1}$, are increasing and continuous on their domains specified in (i).
\ee
\end{lemma}
For each point $x^*\in (f(0), f(b))$,  we are looking for the control parameter $\nu=\nu(x^*)\in (0, 1)$  such that $x^*$ is the fixed  point of the function $g(x):=f(\nu x)$.  Define 
\begin{equation}
\label{def:nux*}
\nu=\nu(x^*):=\Phi(f^{-1}(x^*)).
\end{equation}
\begin{lemma}
\label{lem:fixg}
Let Assumption \ref{as:slope} hold.  For each $x^*\in ( f(0), f(b) )$, 
the control $\nu(x^*)$ defined by \eqref{def:nux*},  satisfies the following conditions:
\be
\item [(i)] $x^*$ is a fixed point of the function $g(x):=f(\nu x)$, i.e. $f\bigl(\nu(x^*)x^*\bigr)=x^*$;
\item [(ii)] $\nu(x^*)\in \left(\Phi(0), \,  \Phi(b)\right)\subset (0, 1)$;
\item [(iii)] $\nu(x^*)$ is an increasing function of $x^*$ on $( f(0), f(b) )$.
\ee
\end{lemma}
\begin{proof}

For each $x^*\in (f(0), f(b))$, we have
\[
\nu(x^*)=\Phi(f^{-1}(x^*))=\frac{f^{-1}(x^*)}{f(f^{-1}(x^*))}=\frac{f^{-1}(x^*)}{x^*},
\]
so
\[
f\bigl(\nu(x^*)x^*\bigr)=f\left(\frac{f^{-1}(x^*)}{x^*} x^*\right)=f\left(f^{-1}(x^*)\right)=x^*,
\]
which implies (i).

Since $x^*\in ( f(0), f(b) )$, 
\[
f^{-1}(x^*)\in (0, b),
\]
and by Lemma \ref{lem:phi-1} we have 
\[
\nu(x^*)\in \left(\Phi(0), \,  \Phi(b)\right)\subset (0, 1),
\]
which proves (ii). Since both $\Phi$ and $f^{-1}$ are increasing functions on $(0, b)$ and $( f(0), f(b) )$, respectively, we conclude that $\nu(x^*)=\Phi(f^{-1}(x^*))$  is increasing as a function of $x^*$ on $( f(0), f(b) )$, which concludes the proof of (iii).
\end{proof}

Finally, let us discuss the role of the inequality $f(b) > b$ in Assumption~\ref{as:slope}. 
If we omit this condition in Assumption~\ref{as:slope},  we should  also consider $b$ such that $f(b)  \leq b$.  
Using the assumptions that  $f(b)  \leq b$ and the  function $\Phi(x)=\frac{x}{f(x)}$ is continuous and monotone increasing on  
$(0,b]$, we get that $\Phi(x)<1$  for small $x>0$ and $\Phi(b) \geq 1$. 
Thus there is the only point $ x^{\ast}\in (0,b]$ where $\Phi( x^\ast)=1$, or $f( x^{\ast})= x^{\ast}$. 
According to Assumption~\ref{as:slope}, the continuous function $f$ is monotone increasing on $[0,  x^{\ast}]$,
$f(x) \geq x$ for $x \in [0,  x^{\ast}]$ and $f(x)<x$ for $x> x^{\ast}$.
By \cite[Lemma 1]{NODY}, see also \cite{Brian},
these conditions guarantee that all solutions of equation \eqref{1} with $x_0>0$
tend to $ x^{\ast}$ as $n \to \infty$. Thus, under $f(b) \leq b$ the unique positive equilibrium $ x^{\ast}$ is stable, 
and there is no need to apply any control to stabilize $ x^{\ast}$. 
This is the reason why we consider only the case when $f(b)>b$.

\section{Multiplicative Perturbations}
\label{sec:multi}

In this section we consider the deterministic PF with variable intensity $\nu_n \in (0,1]$
\begin{equation}
\label{eq:var}
x_{n+1}= f\left( \nu_n x_n \right), \quad x_0>0, \quad n\in \mathbb N,
\end{equation}
and then corresponding stochastic equation \eqref{3} with a multiplicative noise.  For each $x^*\in \bigl(f(0), f(b)\bigr)$,  we establish the control 
$\nu=\nu(x^*)$ and the interval such that 
a solution of \eqref{3} remains in this interval, once the level of noise $\ell$ is small enough.


\subsection{Deterministic multiplicative perturbations}
\label{subsec:detmulti}

In Lemma \ref{lem:multi} below we prove that when  intensity $\nu_n$ belongs to some interval, solution $x_n$ of \eqref{eq:var} will reach and remain  in another interval for big enough $n$.

As discussed in Section \ref{sec:prelim}, Assumption~\ref{as:slope} implies  that  for any  
$\displaystyle \mu  \in \left( \Phi(0), \, \Phi(b)\right)$ 
there  is a $y=\Phi^{-1}(\mu)\in (0,b)$, where $\Phi$ is defined in \eqref{def:Phi} and \eqref{def:phi0}.
\begin{lemma} 
\label{lem:multi}
Let Assumption~\ref{as:slope}  hold,  and $\mu_1$, $\mu_2$ be such that
\begin{equation}
\label{eq:mu_bounds}
\mu_1 \in \left( \Phi(0), \Phi(b) \right) , \quad \mu_2 \in \left(\Phi(0),\mu_1 \right), 
\end{equation}
and, for each $n \in {\mathbb N}$,
\begin{equation}
\label{eq:nu_bounds}
\nu_n \in [\mu_2,\mu_1].
\end{equation}
Then, for any $x_0>0$ and $\varepsilon>0$, there is $n_0 \in {\mathbb N}$ such that 
the solution $x_n$ of equation \eqref{eq:var} for any $n \geq n_0$ satisfies
\[
\nu_n x_n \in \left(\Phi^{-1}(\mu_2)-\varepsilon, \Phi^{-1}(\mu_1)+\varepsilon\right).
\]
\end{lemma}

\begin{proof}
Define
\[
y_1:=\Phi^{-1}(\mu_1), \quad y_2:=\Phi^{-1}(\mu_2).
\]
Lemma \ref{lem:phi-1} implies that, as $\Phi(0)<\mu_2<\mu_1<\Phi(b)$, we have 
\[
0<y_2<y_1<b.
\]
By \eqref{eq:mu_bounds}  and \eqref{eq:nu_bounds},  for each $n\in \mathbb N$, 
\begin{equation}
\label{5}
\Phi(0) < \Phi(y_2) \leq \nu_n \leq \Phi(y_1) <  \Phi(b).
\end{equation}

The proof consists of three main steps. First, we prove that for any $x_0>0$ there is 
an $n_1 \in {\mathbb N}$ such that $\nu_n x_n \in (0,b)$ for $n \geq n_1$. 
Next,  we show that for   any $\varepsilon\in \left( 0, b-y_1\right)$, 
there is  an  $n_2(\varepsilon) \in {\mathbb N}$, $n_2(\varepsilon)>n_1$,  such that  $\nu_n x_n \in (0,y_1+\varepsilon)$, $n \geq n_2(\varepsilon)$.  And, finally, we verify that for   any $\varepsilon\in \left( 0, y_1\right)$, there is  an  $n_3 \in {\mathbb N}$, $n_3>n_2(\varepsilon)$,   such that  $\nu_n x_n \in (y_2-\varepsilon,y_1+\varepsilon)$, $n \geq n_3$.

First, assume that $ \nu_n x_n \geq b$ for some $n \in {\mathbb N}$. Then, \eqref{eq:mu_bounds},
\eqref{eq:nu_bounds} and  Assumption \ref{as:slope} imply $f(\nu_n x_n)/(\nu_n x_n) \leq f(b)/b$ and  
\begin{eqnarray*}
\nu_{n+1} x_{n+1} & = &  \nu_{n+1} f\left( \nu_n x_n \right) 
= \nu_{n+1} \nu_n x_n \frac{f\left( \nu_n x_n \right)}{\nu_n x_n} \\ & \leq & \nu_{n+1}  \nu_n x_n \frac{f(b)}{b} 
\leq   \mu_1 \frac{f(b)}{b} \nu_n  x_n =   \lambda  \nu_n x_n, 
\end{eqnarray*}
where 
$$
\lambda := \mu_1 \, \frac{f(b)}{b} \in (0,1).
$$
We obtained that the positive sequence $\nu_j x_j$, $j \geq n$  does not exceed a geometric sequence:
$\nu_j x_j \leq \lambda^{j-n} \nu_n x_n$ with $\lambda \in (0,1)$, as long as $\nu_{j-1} x_{j-1} \geq b$. If $\nu_k x_k > b$ for the subsequent $k=n,n+1, \dots$,  we get $\nu_j x_j < b$ for
$j>n+[(\ln b - \ln(\nu_n x_n))/\ln \lambda]+1$, where $[t]$ is the integer part of $t$.
Thus for some $n_1 \in {\mathbb N}$,  $x_{n_1}$ satisfies $\nu_{n_1} x_{n_1} < b$. 

Further, to verify that $\nu_j x_j < b$ for any $j \geq 
n_1$,  we note that, by \eqref{5}, $\displaystyle \frac{\Phi(y_1)}{\Phi(b)}=\frac{y_1}{f(y_1)} \frac{f(b)}{b} <1$ and also $\nu_{j+1} \leq \mu_1$. Then, if $\nu_j x_j < b$,  we have
\begin{eqnarray*}
\nu_{j+1} x_{j+1} & = &  \nu_{j+1} f\left( \nu_j x_j\right) \leq \mu_1 f\left( \nu_j x_j\right)\\ &=& \frac{y_1}{f(y_1)} f\left( \nu_j 
x_j\right) < \frac{y_1}{f(y_1)} f(b) = \frac{y_1}{f(y_1)}\, \frac{f(b)}{b}\, b <b.
\end{eqnarray*}
Thus, once $\nu_{n_1} x_{n_1} < b$, all $\nu_j x_j < b$, $j \geq n_1$.

Second, assuming $\nu_n x_n \in (y_1, b)$, we have by \eqref{5} 
$$
\nu_{n+1} x_{n+1} = \nu_{n+1} \nu_n x_n \frac{f\left( \nu_n x_n \right)}{\nu_n x_n} \leq \nu_{n+1} \nu_n x_n \frac{f(y_1)}{y_1} <
\frac{y_1}{f(y_1)} \nu_n  x_n \frac{f(y_1)}{y_1} = \nu_n x_n.
$$
Thus, once some $\nu_n x_n \in (y_1,b)$, we have $\nu_{n+1} x_{n+1} < \nu_n x_n < b$ as well. 

Next, let us take some $\varepsilon\in (0, b-y_1)$.  In the previous computation we apply 
$y_1+\varepsilon$ instead of $y_1$. Then, once $\nu_n x_n \geq y_1+\varepsilon$, we have $x_{n+1} \leq \lambda^{\ast} x_n$, with
$$ 
\lambda^{\ast} :=  \frac{f(y_1+\varepsilon)}{y_1+\varepsilon}\, \frac{y_1}{f(y_1)} = \mu_1 \,
\frac{f(y_1+\varepsilon)}{y_1+\varepsilon} \in (0,1) \,.  
$$
If $\nu_k x_k \geq y_1+\varepsilon$ for the subsequent $k=n,n+1, \dots$  then $\nu_j x_j < y_1 + \varepsilon$ for 
$j>n+[(\ln(y_1+\varepsilon)-\ln(\nu_n 
x_n))/\ln(\lambda^{\ast})]+1$. Thus $\nu_n x_n < y_1 + \varepsilon$ for any $\varepsilon>0$ and 
$n \geq n_2(\varepsilon) \in {\mathbb N}$, for some $n_2(\varepsilon)>n_1$.
 
Let us justify, that, if $\nu_n x_n \leq y_1+\varepsilon<b$, all the subsequent $x_k$ also
satisfy this inequality. If  $\nu_n x_n \leq y_1+\varepsilon<b$ then, by monotonicity of $f$ on $[0,y_1+\varepsilon]$  and since now $\mu_1=\Phi(y_1)<\Phi(y_1+\varepsilon)$, we have
$$
\nu_{n+1} x_{n+1} = \nu_{n+1} f(\nu_n x_n) \leq \nu_{n+1} f(y_1+\varepsilon) < \frac{\nu_{n+1} (y_1+\varepsilon)}{\mu_1} \leq y_1+\varepsilon,
$$
as $\nu_{n+1} \leq \mu_1$. 

Finally, for any $\varepsilon \in (0,y_2)$, if $\nu_n x_n<y_2-\varepsilon$, we obtain
$$
\nu_{n+1} x_{n+1} = \nu_{n+1} \nu_n x_n \frac{f\left( \nu_n x_n \right)}{\nu_n x_n} \geq 
\nu_{n+1} \nu_n x_n \frac{f(y_2-\varepsilon)}{y_2-\varepsilon} \geq
\mu_2 \nu_n x_n \frac{f(y_2 - \varepsilon)}{y_2 - \varepsilon}= \lambda_1 \nu_n x_n,
$$
where
$$
\lambda_1 := \mu_2 \frac{f(y_2 - \varepsilon)}{y_2 - \varepsilon} \in (1,\infty).
$$
Thus, $\nu_{n+1} x_{n+1} \geq \lambda_1 \nu_n x_n$ as long as $\nu_n x_n<y_2-\varepsilon$.
Hence, as in the second part of the proof, there is $N_2(\varepsilon) \in {\mathbb N}$ 
such that $\nu_n x_n > y_2- \varepsilon$ for $n \geq N_2(\varepsilon)$.
If $\nu_n x_n \in (y_2- \varepsilon,b]$ then, as $f$ is monotone on $[0,b]$ and by \eqref{5},
$$
\nu_{n+1} x_{n+1} = \nu_{n+1} f(\nu_n x_n) \geq \mu_2 \frac{f(y_2- \varepsilon)}{y_2- \varepsilon} \,
(y_2- \varepsilon) > \mu_2 \frac{y_2- \varepsilon}{\mu_2} = y_2- \varepsilon.
$$
Thus, for any $x_0>0$ and $\varepsilon>0$, for $n\geq n_3 > \max\{ N_2(\varepsilon),n_2(\varepsilon)\}$  we have $\nu_n x_n \in 
(y_2-\varepsilon,y_1+\varepsilon)$, which concludes the proof.
\end{proof}

\begin{remark}
In Lemma~\ref{lem:multi} we actually proved that,  for small \, $\varepsilon>0$,  \, 
once $\nu_n x_n\in (y_1-\varepsilon, y_2+\varepsilon)$, 
all subsequent $\nu_{n+j} x_{n+j}$, $j \in {\mathbb N}$, are also in this interval. 
This is also true for the results based on Lemma~\ref{lem:multi}, in particular, 
for Lemma~\ref{lem:multistoch} and Theorem \ref{thm:mainmult}.
\end{remark}

\subsection{Stochastic multiplicative perturbations.}
\label{subsec:stochmulti}

Now \, we \, proceed \, to \, stochastic \, equation~\eqref{3}.  
First, we prove a lemma which can be treated as a stochastic version of Lemma \ref{lem:multi} and can be  easily deduced from it.

\begin{lemma}
\label{lem:multistoch}
Suppose that Assumptions~\ref{as:slope} and~\ref{as:chibound} hold, 
$\Phi$ is defined as in \eqref{def:Phi},  $x^*\in (f(0), f(b))$, $\nu=\nu(x^*)$ is defined as in \eqref{def:nux*},
and $\ell\in \mathbb R$ satisfies the inequality
\begin{equation}
\label{def:alphal}
0< \ell <\min\left\{\Phi(b)-\nu, \, \, \nu -  \Phi(0) \right\} .
\end{equation}
Let $x_n$ be a solution to equation~\eqref{3} with $\nu$  satisfying \eqref{def:nux*}, $\ell$ 
satisfying \eqref{def:alphal} and an arbitrary initial value $x_0>0$. 
Then, for any $\varepsilon>0$ there is $N_0 \in {\mathbb N}$ such that, for all $n\ge N_0$,  
\[
\left( \nu+ \ell \chi_{n+1} \right)  x_n \in \left(\Phi^{-1}(\nu-\ell)-\varepsilon, 
\Phi^{-1}(\nu+\ell)+\varepsilon\right), \mbox{~~a.s.}
\]
\end{lemma}
\begin{proof}
From  Lemma \ref{lem:fixg} we conclude that
\[
\nu(x^*)=\Phi (f^{-1}(x^*))\in \left(\Phi(0), \, \Phi(b)\right).
\]
So the right-hand side of  \eqref{def:alphal} is positive.  With $\nu$ defined in \eqref{def:nux*} and 
$\ell$ satisfying  \eqref{def:alphal}, we have, a.s., 
$$\nu_n=\nu+\ell \chi_{n+1} \leq \nu+\ell$$
and
$$
\nu_n=\nu+\ell \chi_{n+1} \geq \nu - \ell,$$
so   $\nu_n=\nu+ \ell \chi_{n+1} \in (\nu-\ell,\nu+\ell)$, a.s.

Now we let $\mu_1:= \nu-\ell$, $\mu_2:= \nu+\ell$ and apply Lemma~\ref{lem:multi}.
\end{proof}

Lemma \ref{lem:multistoch} implies the main result of this section, which states that for each 
$x^*\in (0, f(b))$ we can find a control $\nu$ and a noise level $\ell$, such that  the solution 
eventually reaches and stays in some  interval, defined by $\nu$ and $\ell$, a.s.

\begin{theorem}
\label{thm:mainmult}
Let Assumptions~\ref{as:chibound},~\ref{as:slope} hold, $\Phi$  be defined as in \eqref{def:Phi}, $x^*\in (f(0), f(b))$ be an arbitrary point, 
$\nu=\nu(x^*)$ be defined as in \eqref{def:nux*}, $x_0>0$ and $\ell\in \mathbb R$ satisfy inequality 
\eqref{def:alphal}. Then for the solution $x_n$ of equation~\eqref{3}, the following statements are valid, a.s.
\begin{enumerate}
\item [(i)] For each $\delta>0$ there exists a nonrandom $N=N(\delta)\in \mathbb N$ such that, for all $n\ge N$, 
\begin{equation*}
\label{eq:intmultn}
x_n\in\left (\frac{\Phi^{-1}(\nu-\ell)} {\nu+\ell}-\delta, \,\,  \frac{\Phi^{-1}(\nu+\ell)} {\nu-\ell}+\delta\right).
\end{equation*}
\item [(ii)]
$\displaystyle
\liminf_{n \to \infty} x_n \geq  \frac{\Phi^{-1}(\nu-\ell)} {\nu+\ell}, \quad \limsup_{n \to \infty} x_n \leq \frac{\Phi^{-1}(\nu+\ell)} {\nu-\ell}.
$
\end{enumerate}
\end{theorem}

\begin{proof}
Note that condition \eqref{def:alphal} implies $\nu>\ell$.
Fix $\delta>0$ and take some $\varepsilon>0$ satisfying 
\[
\varepsilon<\delta (\nu-\ell).
\]
By Lemma~\ref{lem:multistoch}, for any $x_0>0$ and $\varepsilon>0$, there is $N_0=N_0(\varepsilon)\in \mathbb N$ such that, a.s.,  
$$
(\nu+ \ell \chi_{n+1}) x_n > \Phi^{-1}(\nu-\ell)-\varepsilon, \quad (\nu+ \ell \chi_{n+1}) x_n 
<\Phi^{-1}(\nu+\ell) +\varepsilon, \quad n \geq N_0.
$$
Since $x_n\ge 0$ and $\nu+ \ell \chi_{n+1}\in (\nu-l, \nu+l)$ we have, a.s.,
\[
\Phi^{-1}(\nu-\ell)-\varepsilon<(\nu+ \ell \chi_{n+1}) x_n\le (\nu+ \ell) x_n
\]
and 
\[
\Phi^{-1}(\nu+\ell)+\varepsilon>(\nu+ \ell \chi_{n+1}) x_n\ge (\nu- \ell) x_n.
\]
Therefore, for $n \geq N_0$ we get, a.s.,
\begin{equation}
\label{ineq:Phie}
\frac{\Phi^{-1}(\nu-\ell)-\varepsilon}{\nu+\ell}<x_n < \frac{\Phi^{-1}(\nu+\ell) +\varepsilon}{\nu-\ell}.
\end{equation}
Since
\[
\frac \varepsilon{v+\ell}<\delta, \quad \frac \varepsilon{v-\ell}<\delta,
\]
from inequality \eqref{ineq:Phie} we obtain  (i).

As $\delta>0$ in (i) is arbitrary,  (i) immediately implies (ii).
\end{proof}

Theorem~\ref{prop:liml0} below deals with the situation when the noise level $\ell$ can be 
chosen arbitrarily small.  It confirms the intuitive feeling that, as the noise level $\ell$ is getting smaller, the solution of  stochastic equation \eqref{2}  behaves similarly to the solution of correspondent deterministic equation \eqref{2}  in terms of  approaching its stable equilibrium $x^*$.

\begin{theorem}
\label{prop:liml0}
Let Assumptions~\ref{as:slope} and~\ref{as:chibound} hold, $x^*\in (f(0), f(b))$ be an arbitrary point,  
$\nu=\nu(x^*)$ be defined as in \eqref{def:nux*}, and $x_0>0$  be an  arbitrary   initial value.  
Then for any   $\varepsilon>0$,  there exists the level of noise $\ell(\varepsilon)>0$ such that 
for each $\ell <\ell (\varepsilon)$, there is a positive integer $N_1=N_1(\varepsilon, \ell, x_0)$ such that for the solution  
$x_n$  of the equation~\eqref{3} with $n\ge N_1$ we have $x_n\in (x^*-\varepsilon, x^*+\varepsilon)$, a.s. 
\end{theorem}

\begin{proof}
For each $\delta>0$,
$x^*\in (f(0), f(b))$, $\nu(x^*)=\Phi(f^{-1}(x^*))$ and $\ell$ satisfying \eqref{def:alphal}, the 
solution  $x_n$ of equation \eqref{3} with an arbitrary initial value  $x_0>0$ satisfies inequality \eqref{ineq:Phie} for $n 
\geq N_0=N_0 (\delta,\ell, x_0)$, which can be written as  
\begin{equation*}
\frac{\Phi^{-1}(\Phi(f^{-1}(x^*))-\ell)-\delta}{\nu+l}<x_n < \frac{\Phi^{-1}(\Phi(f^{-1}(x^*))+\ell) +\delta}{\nu-l}.
\end{equation*}
By continuity of $\Phi$ we have
\[
\lim_{\ell\to 0} 
\frac{\Phi^{-1}(\Phi(f^{-1}(x^*))-\ell)}{\Phi(f^{-1}(x^*))+\ell}=\frac{\Phi^{-1}(\Phi(f^{-1}(x^*)))}{\Phi(f^{-1}(x^*))}
=\frac{f^{-1}(x^*)}{\frac{f^{-1}(x^*)}{x^*}}=x^*
\]
and
\[
\lim_{\ell\to 0} 
\frac{\Phi^{-1}(\Phi(f^{-1}(x^*))+\ell)}{\Phi(f^{-1}(x^*))-\ell}=\frac{\Phi^{-1}(\Phi(f^{-1}(x^*)))}{\Phi(f^{-1}(x^*))}=\frac{f^{-1}(x^*)}{\frac{f^{-1}(x^*)}{x^*}}=x^*.
\]
Then, for each $\delta>0$, we obtain
\begin{equation}
\label{ineq:Phie2}
\lim_{\ell\to 0} 
\frac{\Phi^{-1}(\Phi(f^{-1}(x^*))-\ell)-\delta}{\Phi(f^{-1}(x^*))+\ell}
=\frac{f^{-1}(x^*)-\delta}{\frac{f^{-1}(x^*)}{x^*}}=x^*-\delta 
\frac {x^*}{f^{-1}(x^*)} \end{equation}
and
\begin{equation}
\label{ineq:Phie3}
\lim_{\ell\to 0} 
\frac{\Phi^{-1}(\Phi(f^{-1}(x^*))+\ell)+\delta}{\Phi(f^{-1}(x^*))-\ell}
=\frac{f^{-1}(x^*)+\delta}{\frac{f^{-1}(x^*)}{x^*}}=x^*+\delta 
\frac {x^*}{f^{-1}(x^*)} . 
\end{equation}

Fix some $\varepsilon>0$ as in the statement of the theorem and take
\begin{equation}
\label{def:deltax*}
\delta=\frac \varepsilon 2 \, \, \frac {f^{-1}(x^*)}{x^*}.
\end{equation}
Applying \eqref{ineq:Phie2} and \eqref{ineq:Phie3},
for $\delta$ defined by \eqref{def:deltax*} we can find $\ell(\varepsilon)>0$ such that, for $\ell<\ell(\varepsilon)$, both inequalities below hold:
\begin{equation}
\label{ineq:Phie23}
\begin{split}
&\frac{\Phi^{-1}(\Phi(f^{-1}(x^*))-\ell)-\delta}{\Phi(f^{-1}(x^*))+\ell}>x^*-\delta \frac {x^*}{f^{-1}(x^*)}-\frac \varepsilon 
2=x^*-\varepsilon, \\
& \frac{\Phi^{-1}(\Phi(f^{-1}(x^*))+\ell)+\delta}{\Phi(f^{-1}(x^*))-\ell}<x^*+\delta \frac {x^*}{f^{-1}(x^*)}+\frac \varepsilon 
2=x^*+\varepsilon.
\end{split}
\end{equation}
Fix some $\ell<\ell(\varepsilon)$ and find,  for $\ell$ and $\delta$ (defined by 
\eqref{def:deltax*}),  a   number $N_1=N_1(\varepsilon, \ell, x_0)$ such that, for all $n\ge N_1$,
\begin{equation}
\label{ineq:Phie3a}
\frac{\Phi^{-1}(\Phi(f^{-1}(x^*))-\ell)-\delta}{\Phi(f^{-1}(x^*))+\ell}<x_n 
< \frac{\Phi^{-1}(\Phi(f^{-1}(x^*))+\ell) +\delta}{\Phi(f^{-1}(x^*))-\ell}.
\end{equation}
By applying inequalities \eqref{ineq:Phie23}  to inequality \eqref{ineq:Phie3a} we arrive at
\begin{equation*}
x^*-\varepsilon<x_n <x^*+\varepsilon, \quad n\ge N_1,
\end{equation*}
which concludes the proof.
\end{proof}

\section{Additive Perturbations}
\label{sec:add}

In this section we investigate similar problems for the stochastic equation with additive perturbations, i.e. 
equation \eqref{4}, where $f$ satisfies Assumption \ref{as:slope}. 

Our purpose remains the same: to pseudo-stabilize any point $x^*\in \bigl(f(0), f(b)\bigr)$.  
For each $x^*\in \bigl(f(0), f(b)\bigr)$,  we establish the control $\nu=\nu(x^*)$ and the interval such that 
a solution remains in this interval, once the noise level  $\ell$ is small enough.  For any 
$x^*\in \bigl(f(0), f(b)\bigr)$, the same $\nu=\Phi(f^{-1}(x^*))$ as in the previous section will work. 


\subsection{Deterministic additive perturbations}
\label{subsec:detadd}

Before \, considering \, stochastic difference equation \eqref{4} with an additive noise, 
we first study a deterministic model with variable perturbations
\begin{equation}
\label{4var}
x_{n+1}= \max\left\{ f\left(\nu x_n \right) + r_n, 0\right\}, \quad r_n \in {\mathbb R}, \quad x_0>0, \quad n\in \mathbb N,
\end{equation}
as well as its version without a proportional feedback
\begin{equation}
\label{4var_1}
x_{n+1}= \max\left\{ g(x_n) + r_n, 0\right\}, \quad r_n \in {\mathbb R}\quad x_0>0, \quad n\in \mathbb N.
\end{equation}
The function $g$ can also be viewed as a function in which a control has already 
been incorporated, for example, $g(x)=f(\nu x)$.

\begin{assumption}
\label{as:slopeg}
Assume that the function $g : [0, \infty) \rightarrow [0, \infty)$ is continuous, and
there is a real number $c>0$ such that $g(x)$ is strictly monotone increasing, while the  function 
$\frac{g(x)}{x}$ is strictly monotone decreasing on $(0,c]$, $\frac{g(c)}{c} > \frac{g(x)}{x}$ 
for any $x>c$, and there exists a unique fixed point  $ x^* \in (0, c)$ 
of $g$ on $(0,\infty)$, $g(x^{\ast})=x^*$. 
\end{assumption}

\begin{remark}
\label{rem:tildex*}
 Note that Assumption~\ref{as:slopeg} and Assumption~\ref{as:slope} are the same except the conditions:
 $f(b)>b$ while $g(c)<c.$
\end{remark}

Let  $d>0$ satisfy  
\begin{equation}
\label{add1}
d < \min\left\{ c-g(c), \max_{x\in [0, x^{\ast}]} \left[ g(x)-x\right] \right\}. 
\end{equation} 
Introduce the numbers $y_1$, $y_2$ and $y_3$ as follows:\\
\begin{equation}
\label{add3}
\begin{split}
y_1 :=& \sup \left\{ \left. x \in [0, x^{\ast}] \right| g(x)-x \geq d \right\} \in (0,  x^{\ast}),\\
y_2 :=& \inf \left\{ \left. x \in [ x^{\ast},c] \right| g(x)-x \leq - d \right\} \in ( x^{\ast},c),
\end{split}
\end{equation}
and 
\begin{equation}
\label{add4}
y_3 = \inf \left\{ \left. x \in [ x^{\ast},\infty) \right| g(x)-d \leq y_1 \right\},
\end{equation}
where we assume $y_3=\infty$, if the set in the right-hand-side of \eqref{add4} is empty; see Fig.~\ref{figure1}
for the outline of the points $y_i$, $i=1,2,3$.

\begin{figure}[ht]
\includegraphics[scale=0.8]{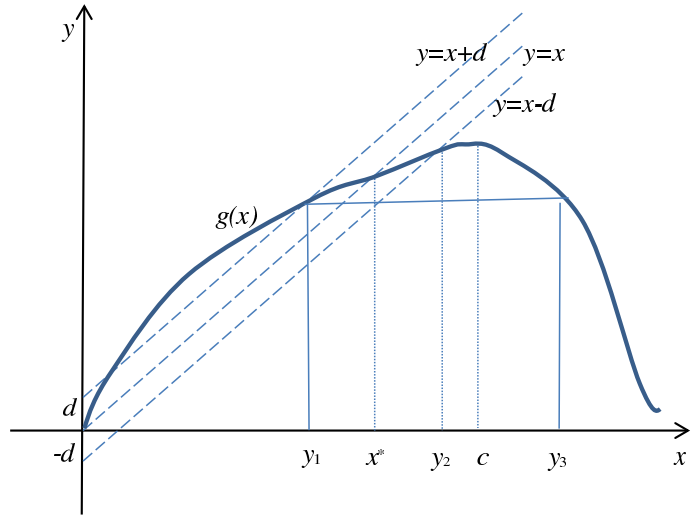}
\caption{The graph of $g(x)$ with $y_i$, $i=1,2,3$, together with the equilibrium $ x^{\ast}$ marked}
\label{figure1}
\end{figure}

\begin{lemma}
\label{lem:y123}
Suppose Assumption \ref{as:slopeg}  holds and  $d$ satisfies \eqref{add1}.  Then
\begin{enumerate}
\item[(i)] the numbers $y_1$, $y_2$ and $y_3$ defined by \eqref{add3} and \eqref{add4}, respectively, exist;
\item [(ii)] $y_1$ and $y_2$ satisfy
\begin{equation}
\label{add2}
0<y_1  <  x^{\ast} <y_2 < c, \quad g(y_1)-d=y_1, \quad g(y_2)+d = y_2,
\end{equation}
$y_2\in ( x^*, c)$ is the only solution of $g(x)+d=x$ on $[ x^{\ast},\infty)$, and $g(x)+d<x$ for $x>y_2$;
\item [(iii)] either finite or infinite $y_3$ satisfies $y_3>c$ and $g(x)-d \geq y_1$ for $x\in [y_1, y_3]$, if $y_3<\infty$,  and for $x\in [y_1, \infty)$ if $y_3=\infty$.
\end{enumerate}
\end{lemma}
\begin{proof}
Note that $g(0) \geq 0$,  $g({x}^*)= x^*$, $g(x)-x$ is continuous and, due to \eqref{add1}, we have
$\max_{x\in [0,  x^{\ast}]} \left[ g(x)-x\right]=g(x_{\max})-x_{\max} > d$ for some $x_{\max}\in [0,  x^{\ast})$.
Thus, $g(x^*)-x^*<d$ and  there is  $x \in (0,  x^{\ast})$ such that $g(x)-x>d$, which  implies $y_1 \in (0,  x^{\ast})$. 

Further, the equations $g(x)=x-d$ and $g(x)=x+d$ have solutions on $(0,c]$.
By \eqref{add1}, $g(c)+d < c$, while  $g( x^{\ast})+d= x^{\ast}+d> x^{\ast}$.
As $g(x)+d$ is continuous, there is a point $y_2 \in ( x^{\ast},c)$ such that $g(y_2)+d=y_2$.
Moreover, $\displaystyle \frac{g(x)+d}{x}$ is a decreasing function on 
$[ x^{\ast},c]$, and for $x>c$, it is less than $(g(c)+d)/c < 1$, since 
\begin{eqnarray*}
\frac{g(x)+d}{x} & = & \frac{g(x)}{x} + \frac{d}{x} < \frac{g(c)}{c} + \frac{d}{c} \\
 & \leq & \frac{g(y_2)}{y_2} + \frac{d}{y_2}
= \frac{g(y_2)+d}{y_2}=1.
\end{eqnarray*}
Thus, there is exactly one fixed point
of $g(x)+d$ on $[ x^{\ast},\infty)$, where the function  $\displaystyle \frac{g(x)+d}{x}$ equals one, 
which is $y_2$, and $g(x)+d<x$ for $x>  x^{\ast}$. 

By \eqref{add3}, we have $y_1\in (0,  x^{\ast})$ and $g(y_1)-d =y_1$. 
Since $g$ increases and is continuous on $[0,c]$, and so is $g(x)-d$,  we have $g(x)-d > y_1$ for any $x \in (y_1,c]$, 
therefore $y_3>c$. 
Relations \eqref{add3} and \eqref{add4}  imply  $g(x)-d \geq y_1$ for $x\in [y_1, y_3]$, if $y_3<\infty$,  and for $x\in [y_1, \infty)$ if $y_3=\infty$.
\end{proof}

\begin{lemma} 
\label{lem:var_gen}
Suppose that Assumption \ref{as:slopeg} holds, $d$ satisfies \eqref{add1},
and $y_1$, $y_2$, $y_3$ are defined by \eqref{add3} and \eqref{add4}, respectively.

Let $x_n$  be a solution of \eqref{4var_1} with $-d \leq r_n \leq d$ and $x_0 \in [y_1,y_3]$.
Then
\begin{enumerate}
\item[(i)] for each $\varepsilon>0$ there exists $N=N(\varepsilon, x_0)$ such that, for $n\ge N$, 
\begin{equation}
\label{add50}
y_1\le x_n\le  y_2+\varepsilon;
\end{equation}
\item[(ii)] we have
\begin{equation}
\label{add5}
\liminf_{n \to \infty} x_n \geq y_1, \quad \limsup_{n \to \infty} x_n \leq y_2.
\end{equation}
\end{enumerate}
\end{lemma}

\begin{proof}
By Lemma \ref{lem:y123}, (iii), if $x_0 \in [y_1,y_3]$ then 
\[
x_1 = g(x_0) + r_0 \geq g(x_0)-d  \geq y_1,
\]
so $x_1\ge y_1$.  By \eqref{add1},  for $x \leq c$,
\[
g(x)+d < g(c)+c-g(c)=c.
\]
By  Lemma \ref{lem:y123}, (ii), for $x_0\ge  x^{\ast}$
\[
g(x_0)+d \leq x_0.
\]
Thus, 
$x_0 \leq y_3$ implies 
$$
x_1 \leq g(x_0)+|r_0| \leq g(x_0)+d \leq \max\{ x_0,c \} \leq y_3.
$$
Similarly, $x_n \in [y_1,y_3]$ yields that $x_{n+1} \in [y_1,y_3]$; consequently,   the first equalities in \eqref{add50}  and \eqref{add5} hold.

For each $x_n\leq y_2$, by monotonicity of $g$ on $[0,c]$ and by \eqref{add3}, 
\[
x_{n+1} \leq g(x_n)+|r_n| \leq g(y_2) + d = y_2.
\]
By  Lemma \ref{lem:y123}, (ii), if $x_n\in (y_2,y_3]$ then $g(x_n)+d <x_n$, and thus
\[
x_{n+1} \leq g(x_n)+|r_n| \leq g(x_n) + d < x_n.
\]
Assuming that all $x_n \in (y_2,y_3]$, we obtain that it is a decreasing sequence which has a limit $s \geq y_2$. 
Thus $x_n<x_1$, and $s>y_2$ implies
\[
x_n-x_{n+1} \geq x_n-g(x_n)- d \geq \min_{x \in [s,x_1]} \left[ x-g(x)-d \right] =\sigma>0.  
\]
Then $\lim_{n\to \infty}x_n=-\infty$, which contradicts to $x_n \geq s>y_2$.

Hence for any $x_0\in (y_2, y_3]$ and $\varepsilon>0$, we need 
a finite number of steps $N=N( \varepsilon, x_0)$ to reach $(y_1, y_2+\varepsilon)$, 
which proves the right inequality in \eqref{add50}.

Thus $\limsup_{n \to \infty} x_n \leq y_2$,
and the second equality in \eqref{add5} is also satisfied, which concludes the proof.
\end{proof}

\begin{lemma} 
\label{lem:add_var}
Let Assumption~\ref{as:slope} hold, $x^*\in (f(0), f(b))$  be an arbitrary point, and $\nu=\nu(x^*)$ be defined as in \eqref{def:nux*}. 
Suppose that  
\begin{equation}
\label{def:gc}
g(x) :=f(\nu x), \quad c:=\frac b{\nu } \, ,
\end{equation}
$d$ satisfies \eqref{add1}, and  $y_1$, $y_2$, $y_3$ are defined in  \eqref{add3} and \eqref{add4}, respectively.

Then all solutions of \eqref{4var} with $-d \leq r_n \leq d$ and $x_0 \in [y_1,y_3]$ satisfy parts (i) and (ii) of Lemma \ref{lem:var_gen}.
\end{lemma}
\begin{proof}
Once we show that $g$ satisfies Assumption \ref{as:slopeg}, the result will follow from Lemma \ref{lem:var_gen}.

From  \eqref{def:Phi} and \eqref{def:nux*} we have
\[
\nu=\Phi\left(f^{-1}(x^\ast) \right)=\frac{f^{-1}(x^\ast)}{x^\ast}\in (0, 1),
\]
and, by Lemma \ref{lem:fixg}, (i),  $g$ has $ x^*$ as a fixed point. 
Also, $g$ is continuous and increasing on $\displaystyle \left[0, c\right]$. For $x>c$,  we have $\nu x>b$, and thus
\[
\frac{g(x)}x=\frac{f\left(\nu x \right)}x = \nu \, \frac{f\left(\nu x \right)}{\nu x} < \nu \, \frac{f(b)}b=\frac {g\left( c \right)}c.
\]
Note that 
\[
\frac {g\left( c \right)}c=\nu \, \frac{f(b)}b=\frac{f^{-1}(x^\ast)}{x^\ast}\, 
\frac{f(b)}b=\frac {\Phi\left(  f^{-1}(x^\ast)\right)}{\Phi(b)}
\]
and $g(x)/x$ is strictly monotone decreasing on $(0,c]$,  since for $0<x_1 < x_2 \leq c$ we have
$$
\frac{g(x_1)}{x_1} = \nu  \frac{f(\nu x_1)}{\nu x_1} > \nu  \frac{f(\nu x_2)}{\nu x_2} = \frac{g(x_2)}{x_2}
$$
by Assumption~\ref{as:slope}, so $\displaystyle \frac{g(x)}{x} < \frac{g(c)}{c}$ for $x>c$.
Since $f^{-1}(x^\ast)<b$ and $\Phi$ is  strictly increasing, we have
\[
\frac {\Phi\left(  f^{-1}(x^\ast)\right)}{\Phi(b)}<1,\quad \text{therefore} \quad \frac {g\left( c \right)}c<1.
\]
 Thus $g(x)<x$ for $x \geq c$, and $g$ cannot have a fixed point on $(c, \infty)$.
Assume that there is one more fixed point $y$ of $g$ on  $(0, c)$ and $y\neq x^{\ast}$. 
Since $y\in (0, c)$, we have $\nu y\in (0, b)$ and 
\[
g(y)=f\left(\nu y \right)=y, \quad \text{or} \quad f\left( \frac{f^{-1}(x^\ast)}{x^\ast}y \right)=y.
\]
Then the fixed point $y$ should be in $(0, f(b))$, and thus $f^{-1}(y)$ is well defined. Therefore 
\[
\frac{f^{-1}(x^\ast)}{x^\ast}y =f^{-1}(y),
\]
which implies
\[
\Phi  \left( f^{-1}(x^{\ast})  \right) =\frac{f^{-1}(x^\ast)}{x^\ast}=\frac{f^{-1}(y)}{y}=\Phi  \left( f^{-1}(y)  \right).
\]
However,  $\Phi  \left( f^{-1}(y)  \right)$ is strictly  increasing  on $(f(0), f(b))$, hence $y=x^{\ast}$.
Application of Lemma~\ref{lem:var_gen} concludes the proof.
\end{proof}

\subsection{Stochastic additive noise}
\label{subsec:stochadd}

In this section, we consider stochastic difference equation \eqref{4}. 
First we state a lemma which will be used in the proof of the main result and which was proved in  \cite{BKR6}. 

\begin{lemma} \cite{BKR6}
\label{cor:barprob}
Let $\xi_1, \dots , \xi_n, \dots$ be a sequence of independent identically distributed random variables 
such that $\mathbb P \left\{\xi\in (a,b)\right\}=\tau\in (0, 1)$ for some interval $(a,b)$, $a<b$. 
Let $\mathcal N_0$  be an a.s. finite random number. Then for each  $J \in {\mathbb N_0}$,
\begin{multline*}
\mathbb{P}\biggl[\text{\rm{There exists} } \mathcal N^*=\mathcal N^*(J)\in (\mathcal N_0, \infty)\\
\text{ \rm{such that}  } \xi_{\mathcal N^*}\in (a,b), \, \xi_{\mathcal N^*+1}\in (a,b), \dots, \xi_{\mathcal N^*+J}\in (a,b)\biggr]=1.
\end{multline*}
\end{lemma}
The next theorem is the main result of this section which states that, for each $\varepsilon>0$, there exists a number  after which a solution of equation \eqref{4} will reach the interval $[y_1, y_2+\varepsilon]$ and stay there forever. However,   in contrast to   Theorem~\ref{thm:mainmult}, where  such a  number is nonrandom and applies to all solutions a.s.,  Theorem \ref{theorem:add} only proves, for each $\omega\in \Omega$,  the existence of  its own $\mathcal N(\omega)$ 
such that  $x_{\mathcal N(\omega)}\in [y_1, y_2+\varepsilon]$ . 
In addition, it is shown  that, with an arbitrarily close to 1 probability, there is a nonrandom 
$N$ such that $x_n\in [y_1, y_2+\varepsilon]$ for $n\ge N$.

\begin{theorem}
\label{theorem:add}
Let Assumptions~\ref{as:slope} and \ref{as:chibound}  hold,  
$x^*\in (f(0), f(b))$  be an arbitrary point, and $\nu=\nu(x^*)$ be chosen as in \eqref{def:nux*},  
$g$ and $c$ be defined as in \eqref {def:gc}, and $d$ satisfy  \eqref{add1}. 
Suppose that $y_1$, $y_2$, $y_3$ are defined as in  \eqref{add3} and \eqref{add4}, respectively,
and $x_n$ is a solution to equation \eqref{4} with an arbitrary $x_0>0$ and $\ell>0$ satisfying 
\begin{equation}
\label{eq:l_min}
\ell \leq d.
\end{equation}
Then
\begin{enumerate}
\item[(i)] for each $\varepsilon>0$, there exists a random $\mathcal N(\omega)=\mathcal N(\omega, x_0, \ell, x^*, \varepsilon)$ such that 
for $n\ge \mathcal N(\omega)$ we have, a.s. on $\Omega$, 
\begin{equation}
\label{add50a0}
y_1\le x_n\le  y_2+\varepsilon;
\end{equation}

\item[(ii)] for each $\varepsilon>0$ and each $\gamma \in (0, 1)$, there exists a nonrandom $N=N(\gamma, x_0, \ell, x^*, \varepsilon)$ 
such 
that
\begin{equation}
\label{add50a}
\mathbb P \{y_1\le x_n\le  y_2+\varepsilon, \,\, \text{\rm{for}} \,\,  n\ge N\}>\gamma;
\end{equation}
\item[(iii)] 
we have
\begin{equation}
\label{add5a}
\liminf_{n \to \infty} x_n \geq y_1, \quad \limsup_{n \to \infty} x_n \leq y_2, \quad \text{a.s}.
\end{equation}
\end{enumerate}
\end{theorem}
 
\begin{proof}
(i) In view of Lemma~\ref{lem:add_var}, we  have to prove first that for the solution $x_n$ of \eqref{4} with an arbitrary initial value 
$x_0>0$, there exists a random a.s. finite $\mathcal N_1(\omega)=\mathcal N_1(\omega, x_0, \ell, x^*)$ such that for $n\ge \mathcal N_1(\omega)$ we have  $x_n \in [y_1,y_3]$, a.s. on $\Omega$. 

The proof consists of two parts.  First, we verify that for $x_0 > y_3$ there exists  a nonnrandom $K\in \mathbb N$ 
such that $x_n \leq y_3$  for $n\ge K$; this part  certainly is worth considering only for $y_3<\infty$. 
We also show that $x_k < y_1$ implies $x_{k+1}<y_3$. Second, we find a random a.s. finite $\mathcal N_1=\mathcal N_1(\omega, x_0, \ell, 
x^*)$ such that  $x_{\mathcal N_1}(\omega) \in  [y_1,y_3]$ and 
a random $\mathcal N(\omega)=\mathcal N(\omega, x_0, \ell, x^*, \varepsilon)$ such that
\eqref{add50a0} holds for $n \geq \mathcal N(\omega)$.

Assume that $y_3<\infty$ and recall that $y_3>c >y_2$, see Lemma~\ref{lem:y123}, (ii),  and Fig.~\ref{figure1}. 

Let $x_k \in (y_3, \infty)$ for some $k\in \mathbb N$. Since $d<c-g(c)$ by \eqref{add1} and $g(x)/x<g(c)/c$ for $x>c$, 
$$ \frac{g(x_k)+\ell \chi_{1}}{x_k} \leq \frac{g(x_k)}{x_k}+\frac{\ell}{x_k}<\frac{g(c)+d}{c} < 
\frac{c-d+d}{c}=1,
$$
so the next sequence term is smaller: $x_{k+1} < x_k$. Moreover, denoting 
$$\lambda:= \frac{g(c)+\ell}{c}<1, $$ we have $x_{k+1} \leq \lambda x_k$, as long as $x_k \in (y_3, \infty)$. 
Thus, whenever $x_0>y_3$, after at most $K$ steps,  where
\begin{equation}
\label{def:K}
K>\ln \left.\left(\frac{y_3}{x_0} \right) \right/ \ln \lambda \, ,
\end{equation}
we have $x_K \in [0,y_3]$. 

If $x_k<y_1$, condition \eqref{add1}, $0<c < y_3$  and the fact that $g$ is monotone increasing on $[0,c]$ imply 
$$
x_{k+1} \leq g(x_k)+d< g(c) +c-g(c)=c< y_3.
$$
Thus, $x_{k+1}<y_3$ for any $x_k<y_1$.  
So we conclude that for each $x_0>0$ and $K$ defined as in \eqref{def:K}, the solution reaches the interval $[0, y_3]$ after at most $K$ 
steps.

Now we proceed to the second part of the proof, assuming that  $x_0< y_1$.
By Assumption~\ref{as:chibound}, for any $\varepsilon_0 \in (0,1)$, the probability $p_1$ of the following event  is positive: 
\begin{equation*}
p_1 := {\mathbb P} \left\{ \chi \in [\varepsilon_0,1] \right\} >0.
\end{equation*}
By \eqref{add2} we have  $y_1<x^{\ast}$, therefore  $x_k \in [0,y_1)$ implies $g(x_k) > x_k$.
If in addition $\chi_{k+1} \in [\varepsilon_0,1]$, we have
$$x_{k+1} = g(x_k)+\ell \chi_{k+1} >x_k+\ell \varepsilon_0.$$
Fix $\varepsilon_0 \in (0,1)$ and define
$$
J := \left[ \frac{y_1}{\ell \varepsilon_0} \right] +1.
$$
If $J$ successive $\chi_j$, starting with $\chi_{k+1}$, satisfy  $\chi_{j} \in [\varepsilon_0,1]$, we have either $x_{k +j}\geq y_1$, for some $j=1, 2, \dots, J-1$,  or
$$x_{k+J} > x_k + J \ell  \varepsilon_0 > y_1. 
$$
By Lemma~\ref{cor:barprob}, a.s., there is a sequence of such $J$ successive $\chi_j$, which starts from an a.s. finite random number $\mathcal N^*$.

Summarizing all of the above,  we conclude that  for $x_0>y_3$  after $n\le K$ steps we have $x_n\in [0, y_3]$. If $x_n<y_1$ we need at 
most $\mathcal N^*+J$ steps to reach $[y_1, y_3]$. Once $x_n \in [y_1, y_3]$,  all $x_{n+k} \in [y_1, y_3]$, $k \in \mathbb N$.  Thus,  by Lemma \ref{cor:barprob}, for 
\begin{equation}
\label{def:N1}
n\ge \mathcal N_1=K+\mathcal N^*+J
\end{equation}
 we have $x_{n} \in [y_1,y_3]$, a.s.  
 
Fix some $\varepsilon>0$. Suppose that $x_n\in [y_2+\varepsilon, y_3]$, where $n$ satisfies \eqref{def:N1}, and apply Lemmata~\ref{lem:var_gen}-\ref{lem:add_var}, considering $x_n=x_{n(\omega)}$ as the initial value for each  $\omega\in \Omega$.  Then, for each  $\omega\in \Omega$,  Lemmata~\ref{lem:var_gen}-\ref{lem:add_var} claim the existence of an a.s. finite number $\mathcal N_2(x_{n(\omega)}, \varepsilon)$  such that  the solution will be in $[y_1, y_2+\varepsilon]$ after $\mathcal N_2$ steps.  Denoting 
\[
\mathcal N=\mathcal N(x_0, \varepsilon, \ell, x^*):=K+\mathcal N^*+J +\mathcal N_2,
\] 
we conclude that 
\begin{equation}
\label{eq:xn12}
x_n \in [y_1,y_2+\varepsilon] \quad  \text{for any $n \geq \mathcal N$},
\end{equation}
which completes the proof of (i).

Since $\mathcal N(x_0, \varepsilon, \ell, x^*)$ is a.s. finite, we can define
\[
 \Omega_i:=\{\omega\in \Omega: i-1\le \mathcal N<i\}.
\]
Note that $\Omega=\bigcup_{i=1}^\infty \Omega_i$ and $\Omega_i\bigcap \Omega_j=\emptyset$ for all $i\neq j$, 
therefore $\sum\limits_{i=1}^\infty \mathbb P(\Omega_i)=1$.

For any $\gamma\in (0, 1)$, we define  $\Omega(\gamma):= \bigcup_{i=1}^N \Omega_i$, where $N=N(\gamma, x_0, \ell, x^*, 
\varepsilon)\in \mathbf N$ is such that
\[
\sum_{i=1}^N \mathbb P(\Omega_i)\in (\gamma, 1).
\]
Then,
\[
\mathcal N(x_0, \varepsilon, \ell, x^*)\le N(\gamma, x_0, \ell, x^*, \varepsilon) \quad \text{on}\quad \Omega(\gamma), \quad \text{and}\quad \mathbb P(\Omega_\gamma)>\gamma.
\]

This completes  the proof of (ii). Part (iii) follows from \eqref{eq:xn12} since $\varepsilon >0$ is arbitrary.
\end{proof}
The next theorem deals with the situation when the level of noise can be decreased. It shows that a solution will eventually be in any arbitrarily small neighborhood of the point $x^{\ast}$ with arbitrarily close to 1 probability.

\begin{theorem}
\label{prop:l0add}
Let Assumptions~\ref{as:slope} and \ref{as:chibound}  hold,  $x_0>0$ be an arbitrary initial value, $x^*\in (f(0), f(b))$  be an arbitrary point, $\nu=\nu(x^*)$ be chosen as in \eqref{def:nux*}.

Then, for each $\varepsilon>0$ and $\gamma\in (0, 1)$, we can find $d$ such that for the solution $x_n$ to \eqref{4} with  $\ell\le d$,   and for  some  
nonrandom  $N=N(\gamma, x_0, \ell, x^*, \varepsilon)\in \mathbb N$, we have 
\[
\mathbb P\{x_n\in (x^*-\varepsilon, x^*+\varepsilon) \,\,  \text{for all $n\ge N$}\}\ge \gamma.
\]
\end{theorem}

\begin{proof}
Let  $g$ and $c$ be defined as in \eqref {def:gc} and let $x_{\max}$ be a point (maybe not unique) where the function $g(x)-x$ attains its maximum on $[0, x^*]$:
\[
g(x_{\max})-x_{\max}=\max_{x\in [0, x^*]}\{ g(x)-x\}.
\]
Note that, for each $d$ satisfying \eqref{add1}, the values $y_1$ and $y_2$ defined by \eqref{add3} satisfy
\[
y_1\in (x_{\max}, x^*), \quad y_2\in (x^*, c),
\]
and $y_1=y_1(d)$, $y_2=y_2(d)$. Let us show that 
\[
\lim_{d\to 0}y_1(d)=\lim_{d\to 0}y_1(d)=x^*.
\]
To this end, define 
\[
\Psi(x)=\frac {g(x)}x, \quad x>0.
\]
By Assumption \ref{as:slopeg}, the  function $\Psi$ decreases and is continuous on $(0, c)$, so it has an inverse function 
$\Psi^{-1}$, which also decreases and is continuous on $(\Psi(0), \Psi(c))$.  We have  $\Psi(x^*)=1$ and 
\[
g(y_1)-y_1=d, \quad y_2-g(y_2)=d.
\]
So, by dividing each of the above equations by $y_1$ and $y_2$, respectively, we arrive at
\[
\frac{g(y_1)}{y_1}-1=\frac d{y_1}, \quad 1-\frac{g(y_2)}{y_2}=\frac d{y_2}.
\]
This leads to the estimates
\[
1<\Psi(y_1)=1+\frac d{y_1}<1+\frac d{x_{\max}}\to 1 \quad \text{as} \quad d\to 0
\]
and 
\[
1>\Psi(y_2)=1-\frac d{y_2}>1-\frac d{x^*}\to 1 \quad \text{as} \quad d\to 0.
\]
Hence
\begin{equation}
\begin{split}
\label{limlrx*}
x^*=&\Psi^{-1}(1)>y_1(d)>\Psi^{-1}\left(1+\frac d{x_{\max}}\right)\to x^* \quad \text{as} \quad d\to 0,\\
x^*=&\Psi^{-1}(1)<y_2(d)<\Psi^{-1}\left(1-\frac d{x^*}\right)\to x^*\quad \text{as} \quad d\to 0.
\end{split}
\end{equation}
Now,  fix some $\varepsilon>0$ and find $\delta$ such that
\begin{equation}
\label{cond:contPhi}
\left|\Psi^{-1}\left(1+u\right)-\Psi^{-1}\left(1\right)\right|<\frac{\varepsilon}{2} \quad \text{as soon as} \quad |u|<\delta.
\end{equation}
Let $d$, in addition to \eqref{add1}, satisfy
\[
d<\delta x^*.
\]
Then, from \eqref{limlrx*} and \eqref{cond:contPhi}, we obtain that
\begin{equation}
\begin{split}
\label{limy12x*}
y_1(d)>&\Psi^{-1}(1)+\left[
\Psi^{-1}\left(1+\frac d{x_{\max}}\right)-\Psi^{-1}(1)\right]\ge x^*-\frac{\varepsilon}{2},\\
y_2(d)<&\Psi^{-1}(1)+\left[
\Psi^{-1}\left(1+\frac d{x^*}\right)-\Psi^{-1}(1)\right]\ge x^*+\frac{\varepsilon}{2}.
\end{split}
\end{equation}
Further, we work only with $\ell\le d$.
Applying part (ii) of Theorem~\ref{theorem:add} for 
$\varepsilon/2$ and $\gamma\in (0,1)$, we find a nonrandom $N=N(\gamma, x_0, \ell, x^*, \varepsilon)\in \mathbb N$ and 
$\Omega_\gamma\subset \Omega$ with  $\mathbb P (\Omega_\gamma)>\gamma$, such that for each $\omega\in \Omega_\gamma$ and $n\ge N$ 
we have
\[
y_1\le x_n\le y_2+\frac{\varepsilon}{2}.
\]
This, along with inequalities \eqref{limy12x*}, gives us the desired result: 
for each  $n\ge N$, we have on ~$\Omega_\gamma$ 
\[
x^*- \frac{\varepsilon}{2}< y_1\le x_n\le y_2+ \frac{\varepsilon}{2}<x^*+\varepsilon,
\]
which concludes the proof.
\end{proof}


\section{Numerical Examples}
\label{sec:examples}

In this section, we consider two examples: the Beverton-Holt chaotic map
\begin{equation}  
\label{BH_NODY}
f(x)= \frac{2.5 x}{1+x^5}, \quad x\ge0, 
\end{equation}  
considered in \cite{NODY}
and a particular case of \eqref{eq:milton}
\begin{equation}
\label{eq:CAMWA}
f(x) = x \left( 0.55 + \frac{3.45}{1+x^9} \right), \quad x\ge0, 
\end{equation}
introduced in \cite{Milton} and also considered in \cite{Liz_CAMWA}.

\begin{example}
We can show that the function $f$ defined in \eqref{BH_NODY} is increasing, 
up to its maximum attained at $b=2^{-2/5}\approx 0.757858$. Also, 
\[
f(b)=f(2^{-2/5})=2 \times  2^{-2/5}\approx 1.5157, 
\]
and 
\[
\Phi(0)=\lim_{x\to 0}\frac x{f(x)}=\lim_{x\to 0}\frac{1+x^5}{2.5}=0.4, \quad \Phi(b)=\frac b{f(b)}=\frac{1+2^{-2}}{2.5}=0.5.
\]
Based on the results of Sections \ref{sec:multi} and \ref{sec:add} we conclude that with PF control we can stabilize any point 
$x^* \in (f(0),f(b))=(0, 2b)\approx (0, 1.5157)$ with  corresponding $\nu$-values belonging  to the interval $(\Phi(0), \Phi(b))=(0, 0.5)$.

For example, we can stabilize $x^*=1.5$, which can be achieved for $\nu \approx 0.4685$.
In Fig.~\ref{figure2}, we get a blurry equilibrium centered at $x^*=1.5$ with
smaller $\ell=0.01$ (upper left) and larger $\ell=0.025$ (upper right).
In the other graphs $\ell=0.015$ and five runs for each parameter set are presented, where we stabilize
$x^*=1.125$ ($\nu \approx 0.415033$) and $x^*=1.1$ ($\nu \approx 0.40721$) in the medium row
and present the two cases where all solutions tend to zero ($\nu =0.39<0.4$, lower left) or are not stabilized
($\nu =0.75>0.5$, lower right) in the lower row.

\begin{figure}[ht]
\centering
\includegraphics[height=.18\textheight]{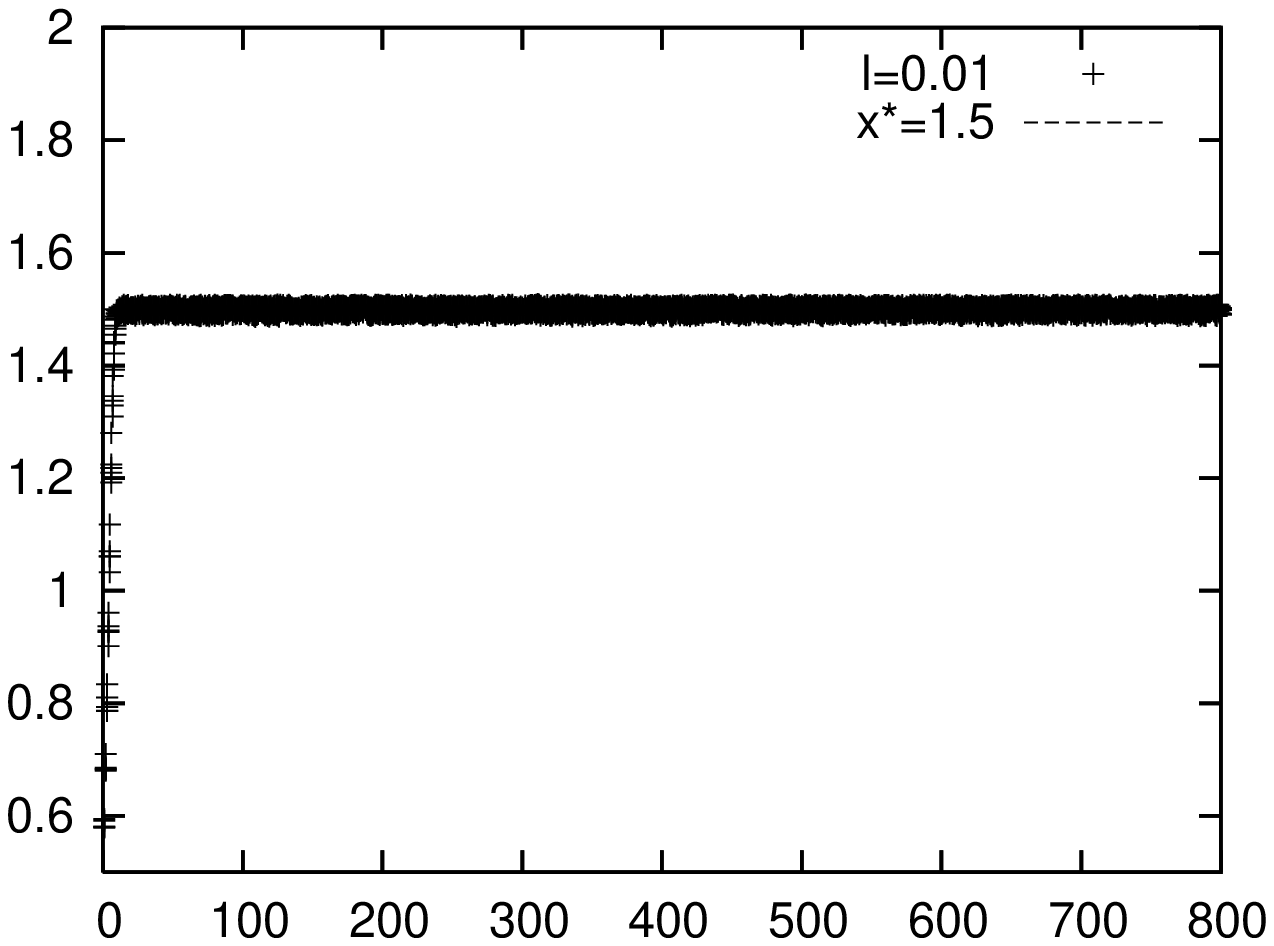}
\hspace{10mm}
\includegraphics[height=.18\textheight]{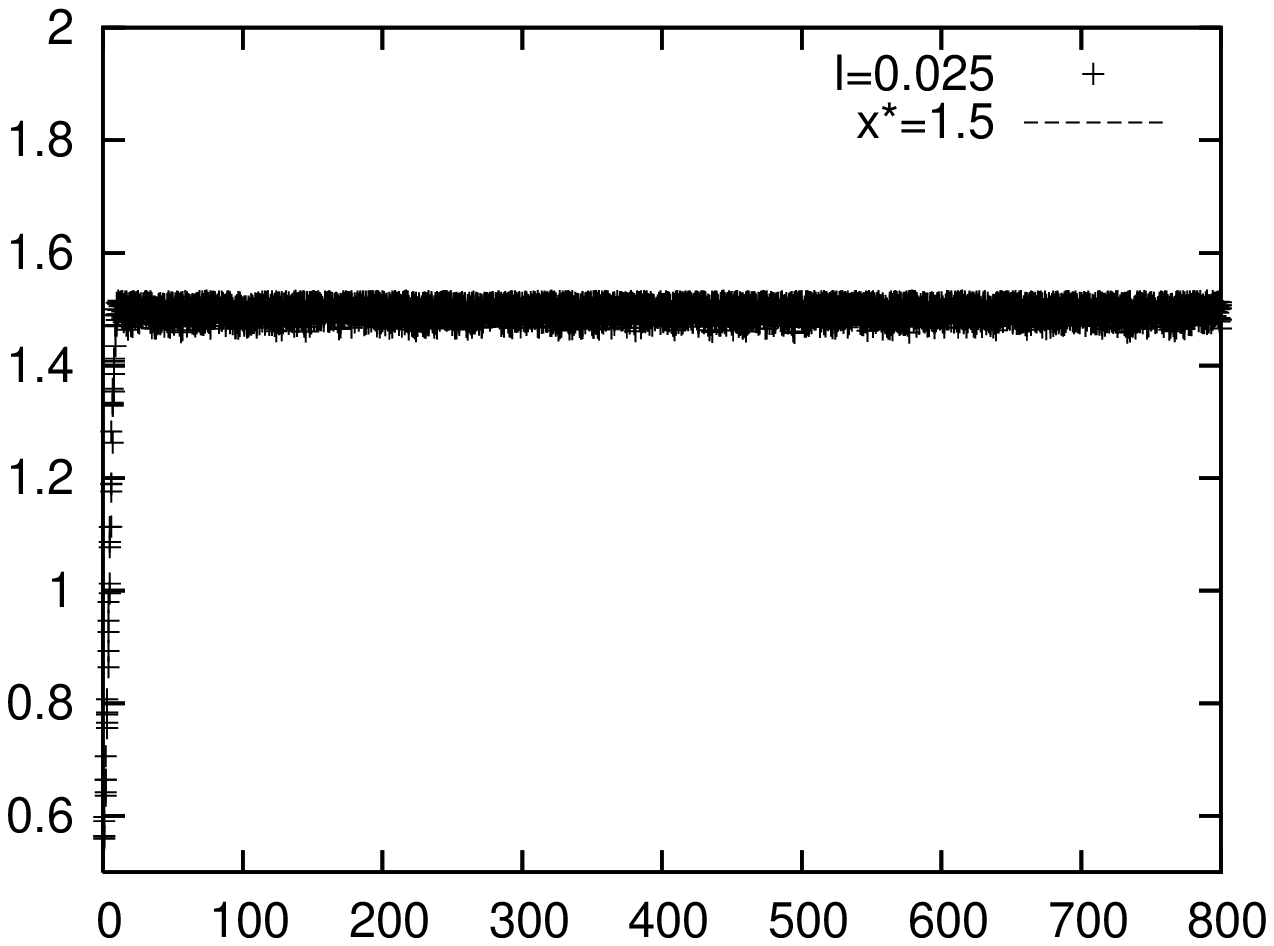}
\vspace{2mm}
\includegraphics[height=.18\textheight]{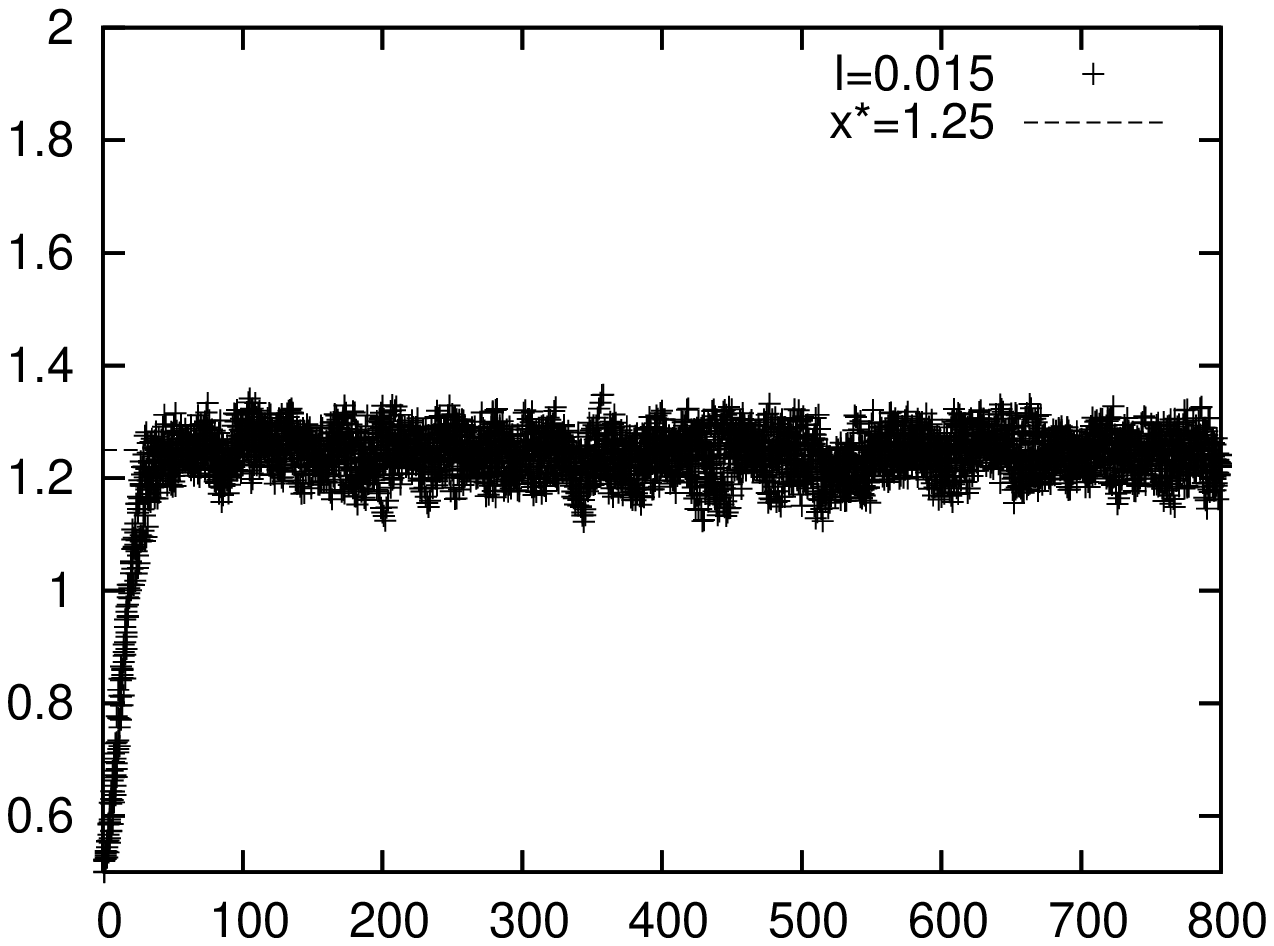}
\hspace{10mm}
\includegraphics[height=.18\textheight]{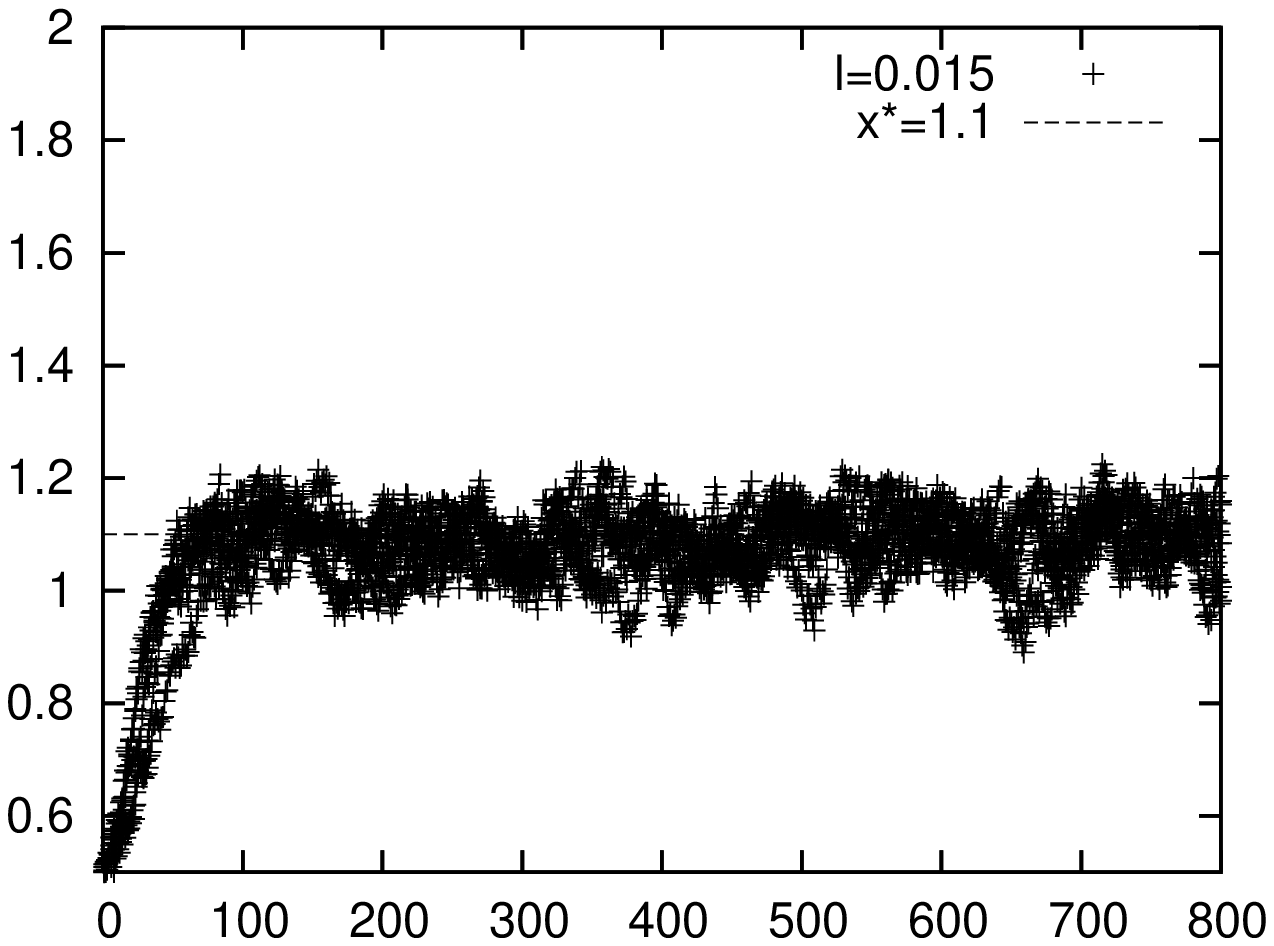}
\vspace{2mm}
\includegraphics[height=.18\textheight]{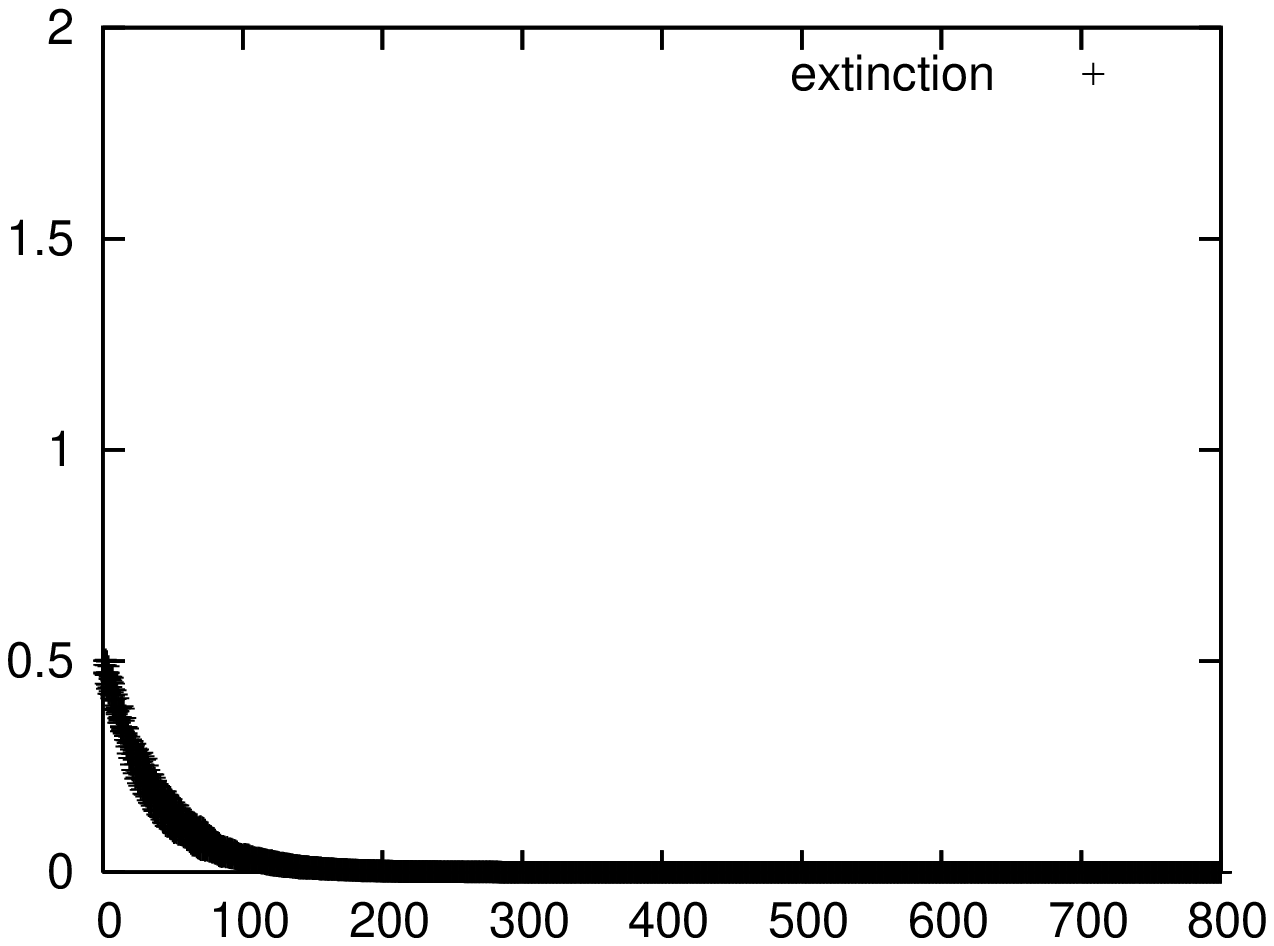}
\hspace{10mm}
\includegraphics[height=.18\textheight]{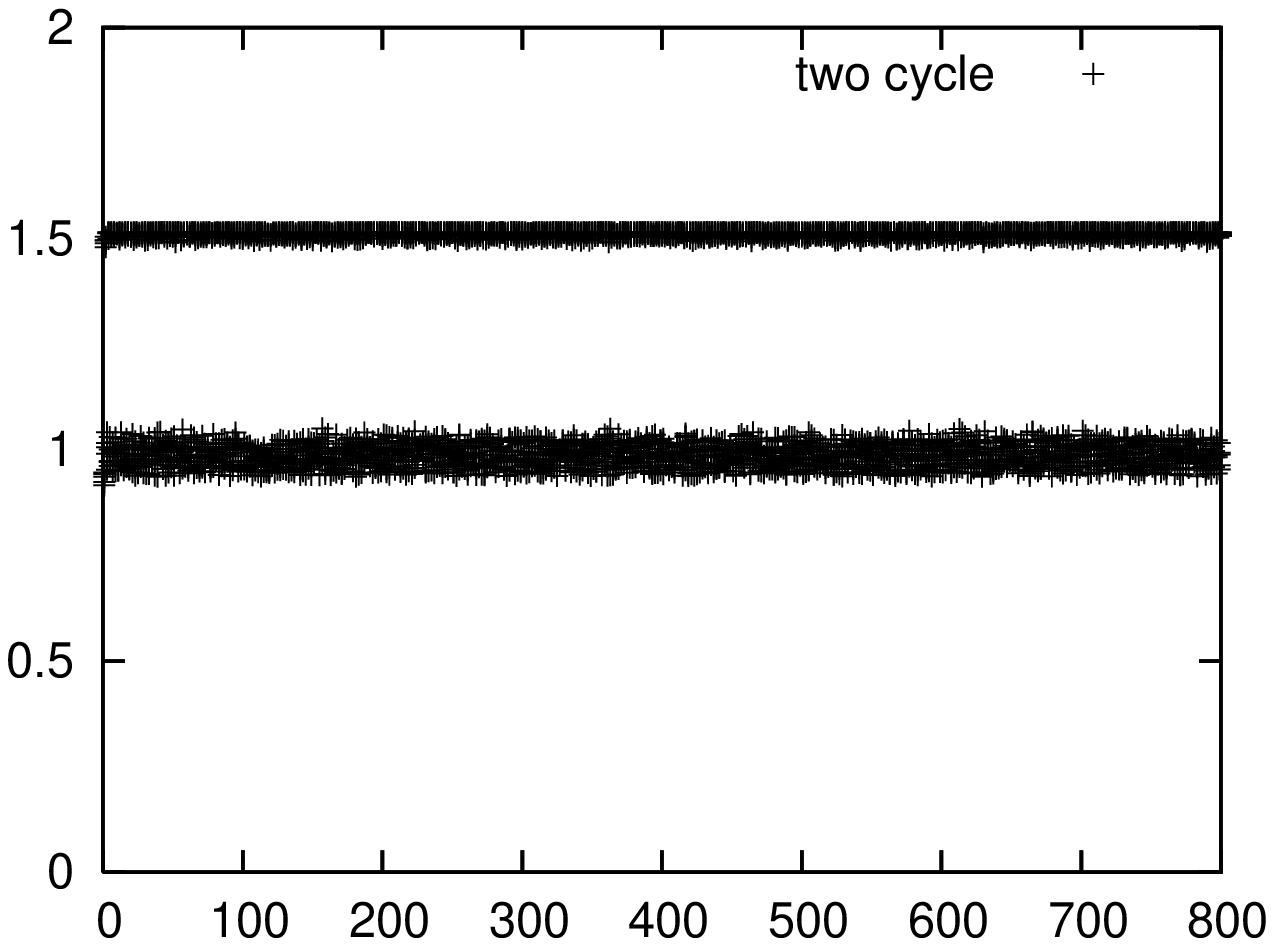}
\caption{Solutions of the difference equation with $f$ as in
(\protect{\ref{BH_NODY}}) and
multiplicative stochastic perturbations with $\ell=0.01$ (upper left), 
$\ell=0.025$ (upper right), where PF control aims at stabilizing $x^*=1.5$, $\nu \approx 0.4685$
and $\ell=0.015$ (two lower rows), with either $x^*=1.125$ (second row,left) or $x^*=1.1$ stabilized (second row, right), and 
$x^*=0$ is stabilized for $\nu =0.39$ (lower left); for $\nu =0.75$ there is no blurred equilibrium but oscillations
(lower right). Everywhere $x_0=0.5$.}
\label{figure2}   
\end{figure}

In Fig.~\ref{figure3} the same runs are presented for the case of the additive noise.

\begin{figure}[ht]
\centering
\includegraphics[height=.18\textheight]{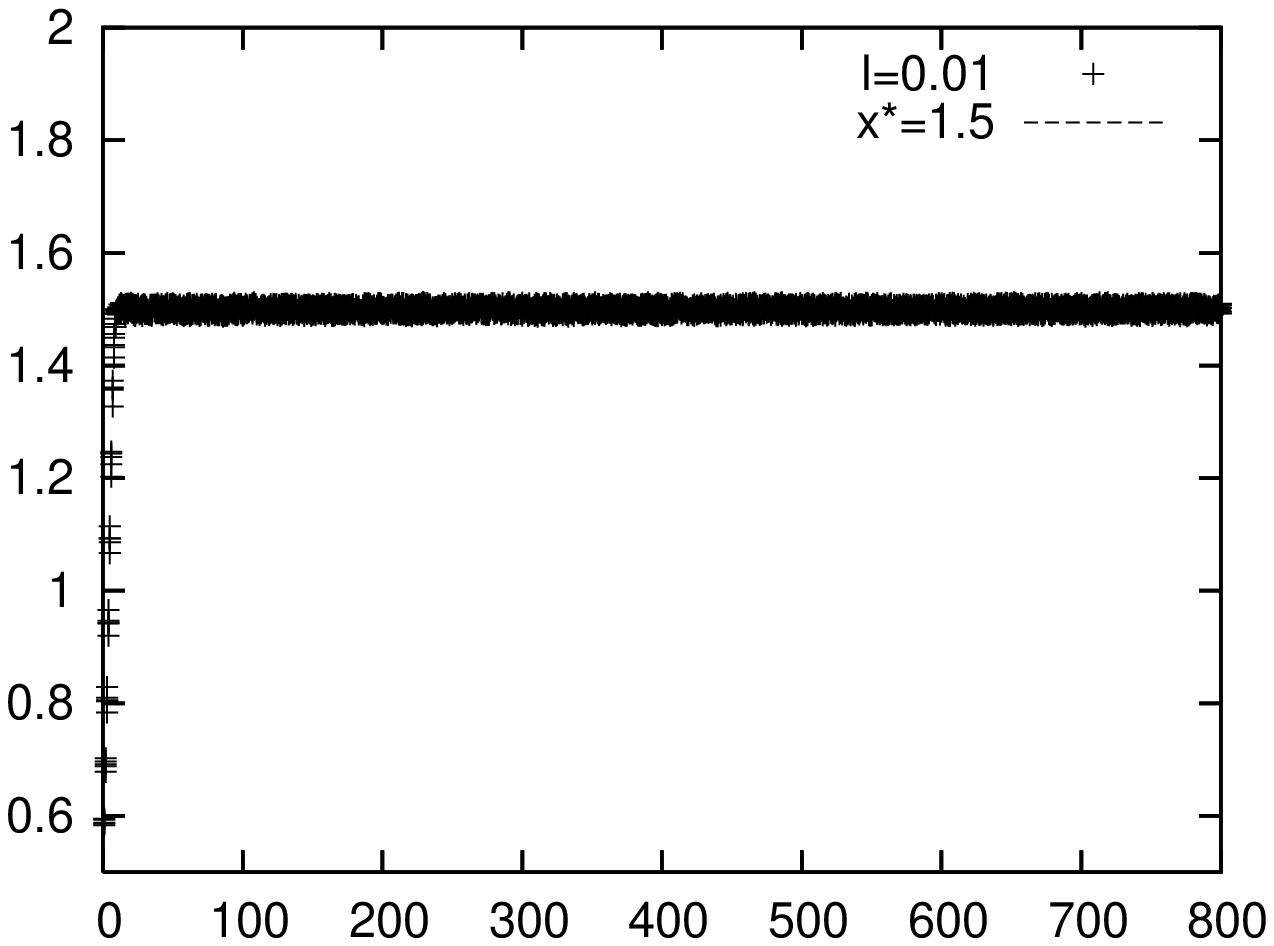}
\hspace{10mm}
\includegraphics[height=.18\textheight]{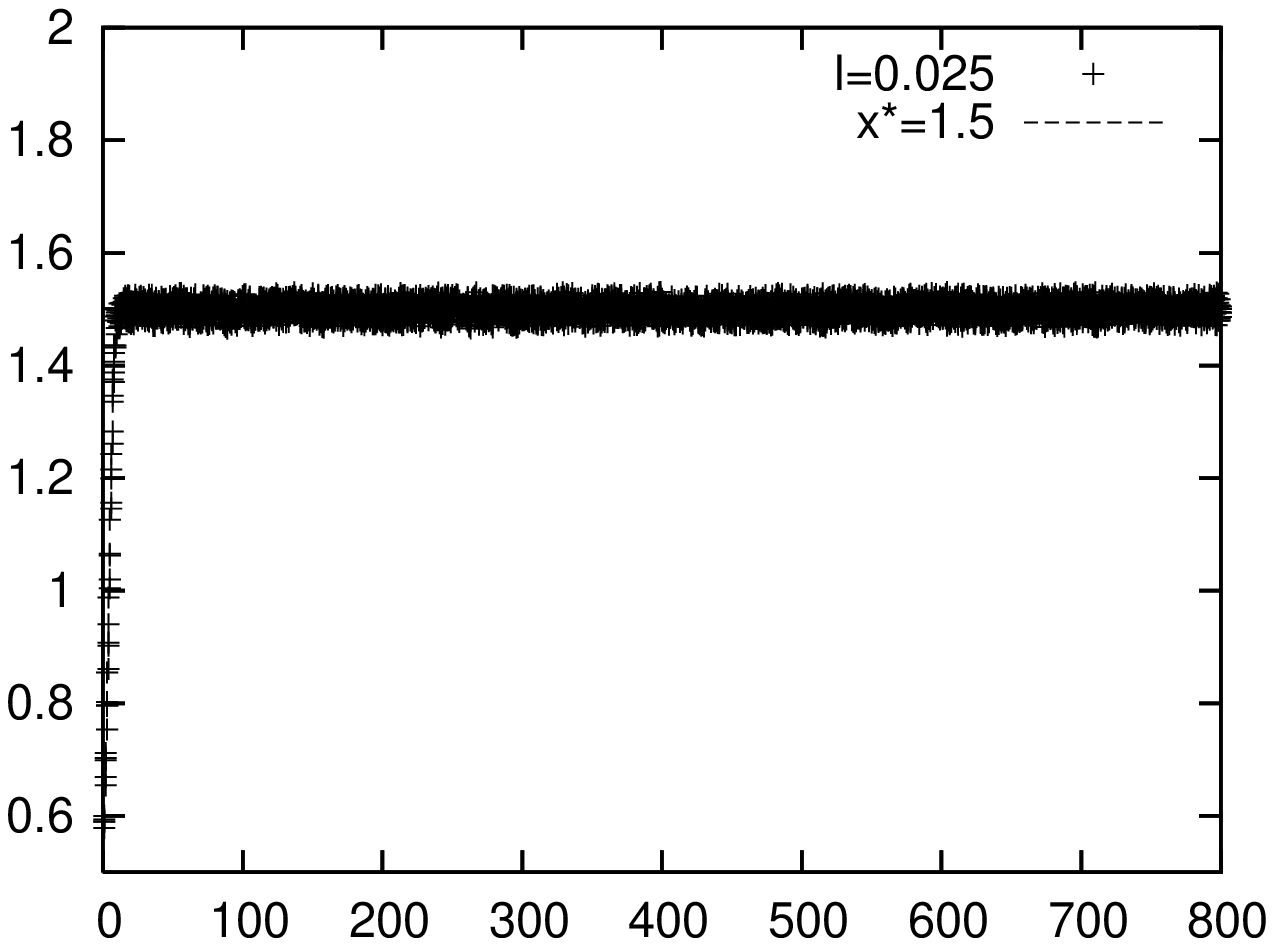}
\vspace{2mm}
\includegraphics[height=.18\textheight]{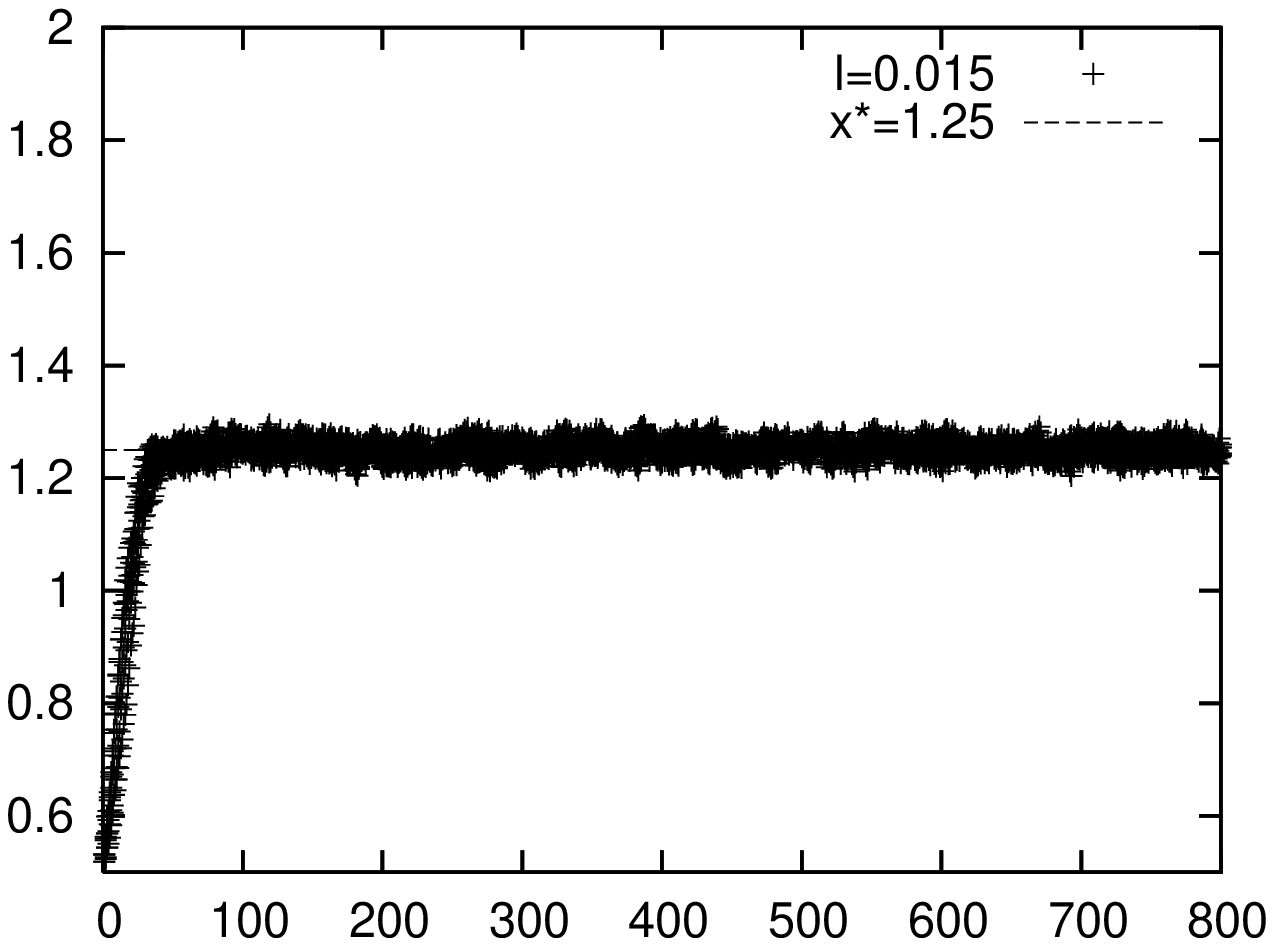}
\hspace{10mm}
\includegraphics[height=.18\textheight]{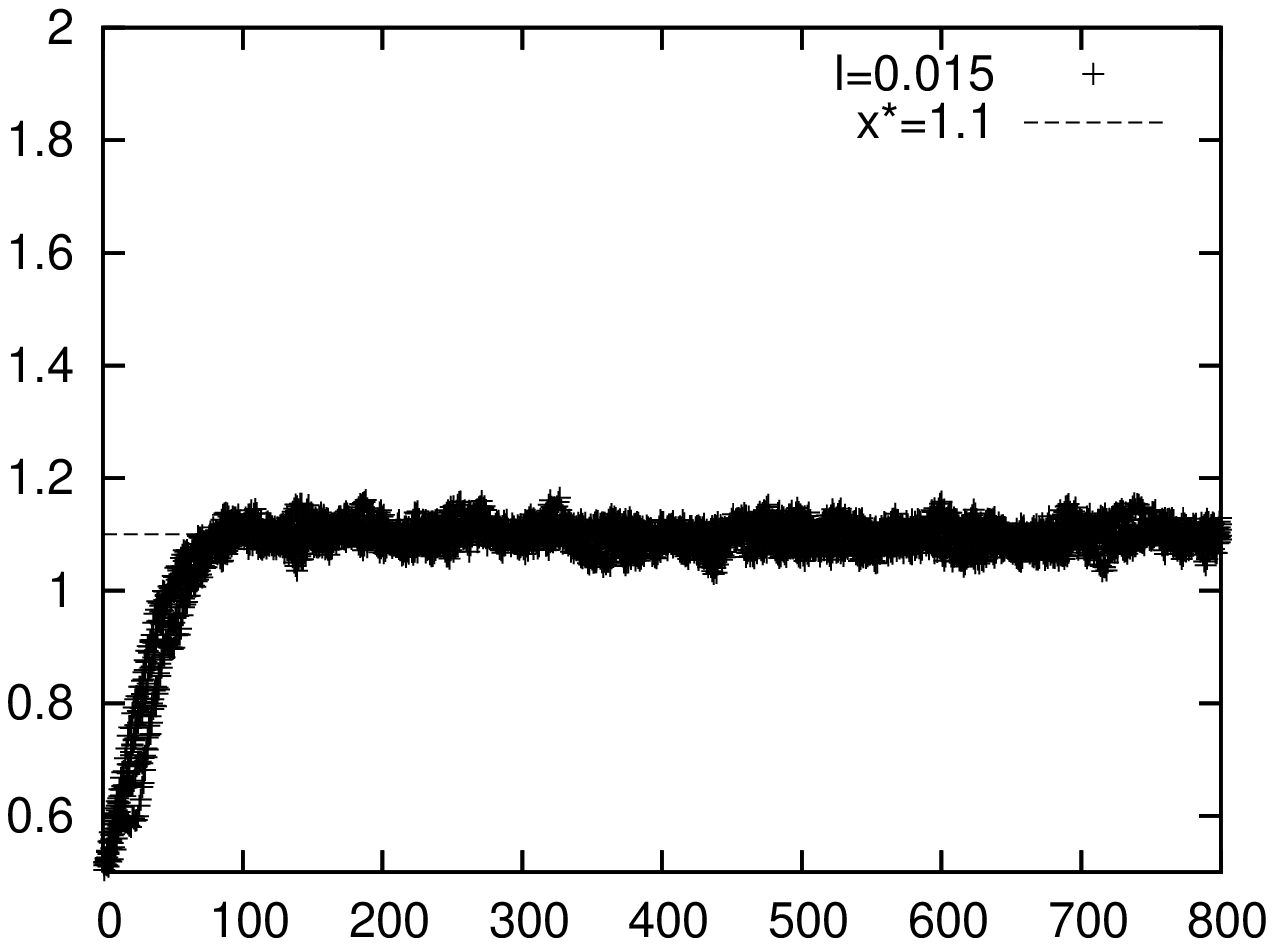}
\vspace{2mm}
\includegraphics[height=.18\textheight]{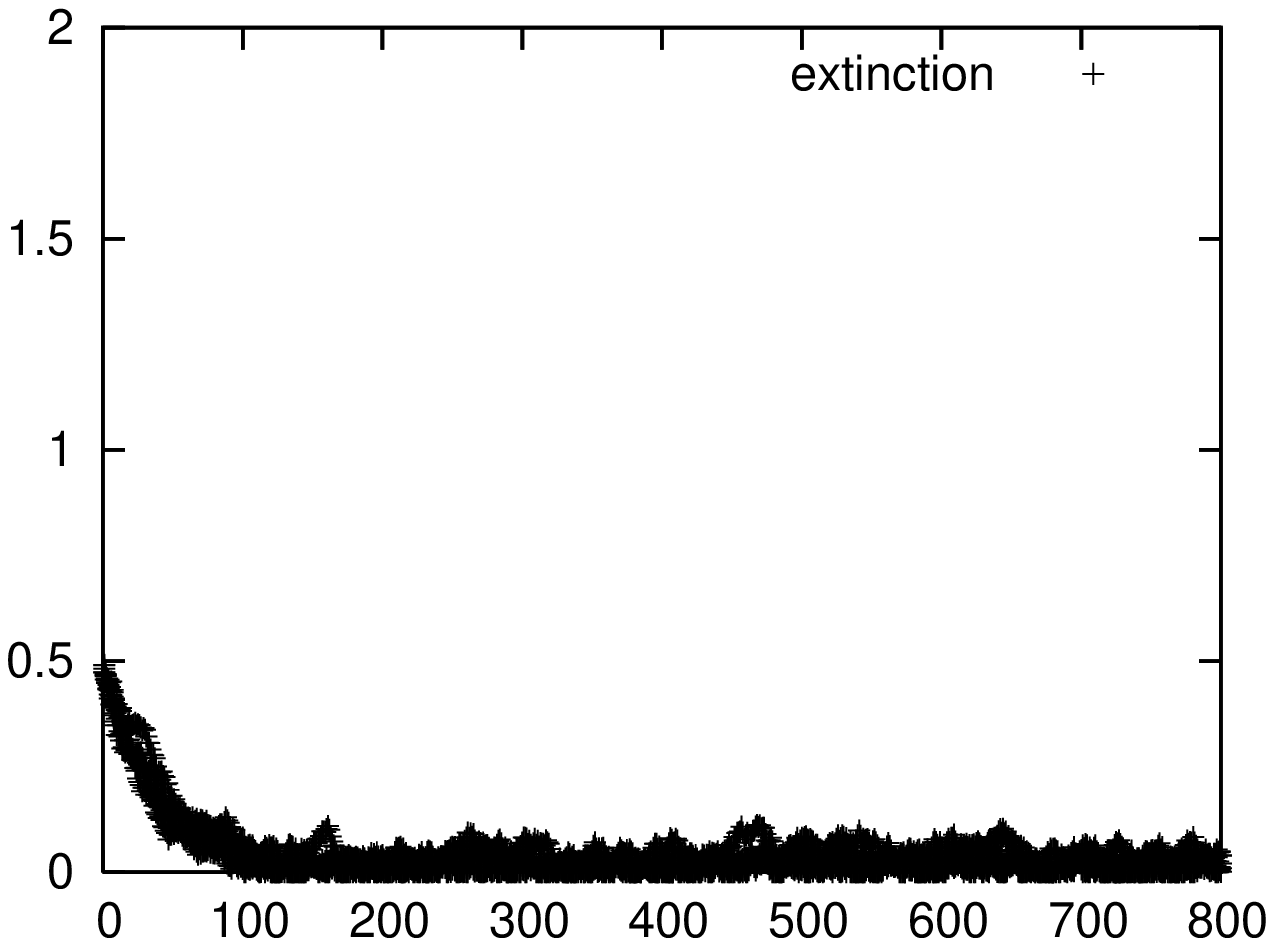}
\hspace{10mm}
\includegraphics[height=.18\textheight]{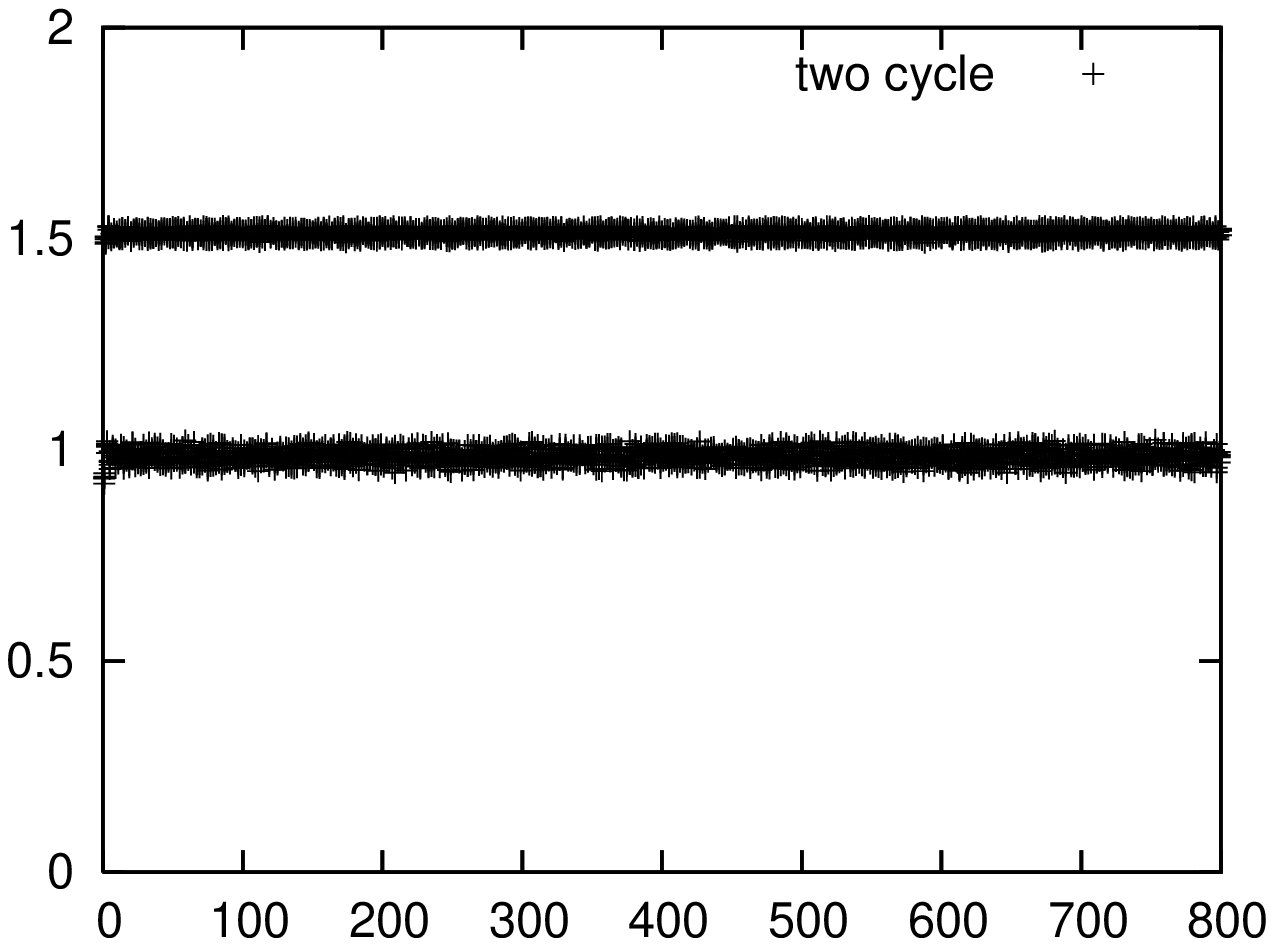}
\caption{Solutions of the difference equation with $f$ as in~(\protect{\ref{BH_NODY}}) and
additive stochastic perturbations with $\ell=0.01$ (upper left), 
$\ell=0.025$ (upper right), where PF control aims at stabilizing $x^*=1.5$, $\nu \approx 0.4685$
and $\ell=0.015$ (two lower rows), with either $x^*=1.125$  (second row,left) or $x^*=1.1$ stabilized (second row, right), and 
$x^*=0$ is stabilized for $\nu =0.39$ (lower left); for  $\nu =0.75$ there is no blurred equilibrium but oscillations
(lower right). Everywhere $x_0=0.5$.}
\label{figure3}   
\end{figure}
\end{example}

\begin{example}
Let us apply PF control to  \eqref{eq:CAMWA}. 
We can  stabilize $x^*=2.5$, which can be achieved for $\nu \approx 0.253555$.
The dependency of the solution variation on $\ell$ is illustrated in Fig.~\ref{figure4},
for $\ell=0.01$ (left) and $\ell=0.025$ (right). 
In Fig.~\ref{figure5}, left, we stabilize the maximum $\approx 2.877$ of the function in the right-hand side of
\eqref{eq:CAMWA}, with $\nu \approx 0.282$, in the middle, with $\nu=0.23$, the zero equilibrium is stabilized,
while the right figure corresponds to the blurred cycle with $\nu=0.35$.

\begin{figure}[ht]
\centering
\includegraphics[height=.2\textheight]{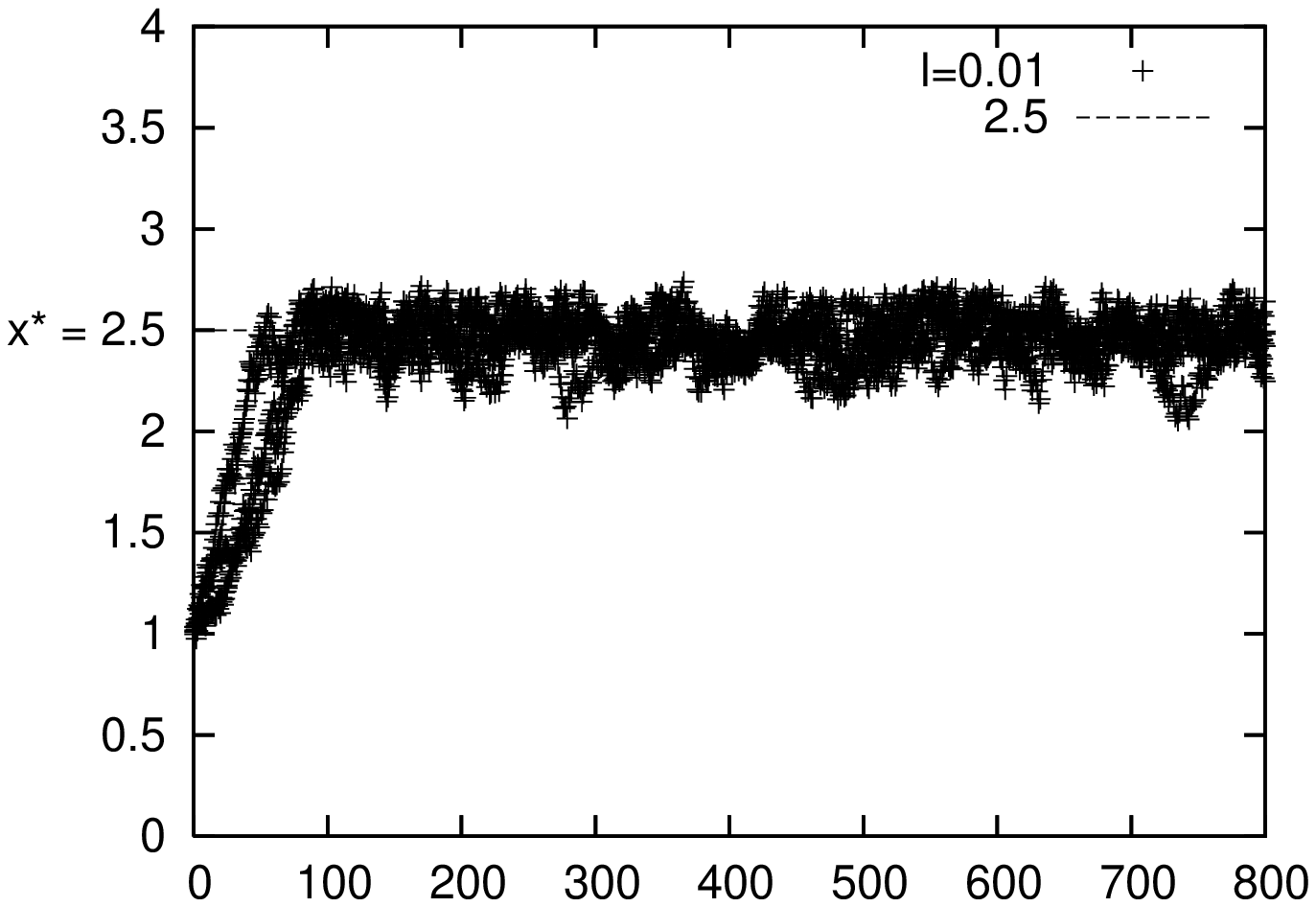}
\hspace{10mm}
\includegraphics[height=.2\textheight]{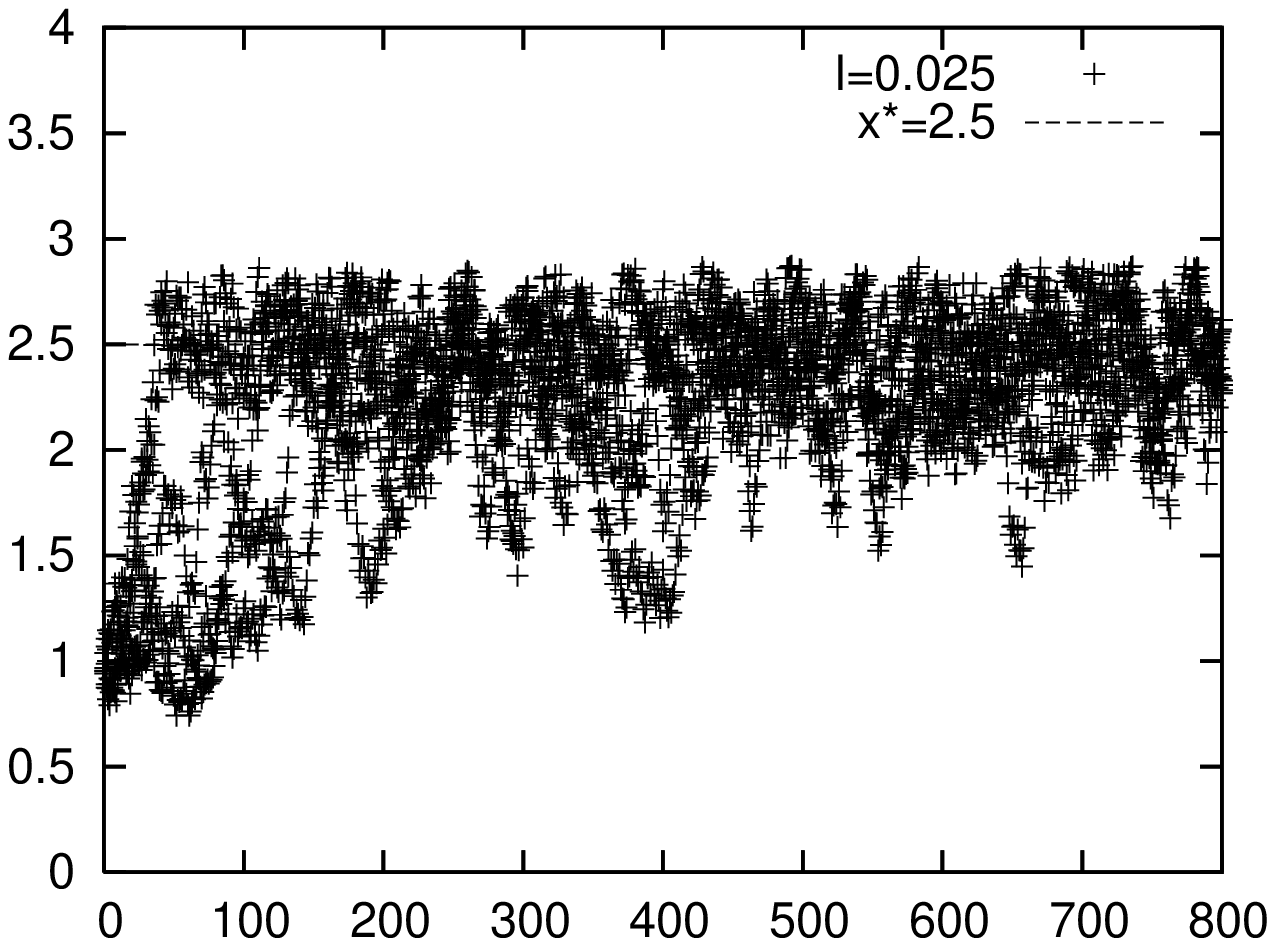}
\caption{Solutions of the difference equation with $f$ as in
(\protect{\ref{eq:CAMWA}}) and
multiplicative stochastic perturbations with $\ell=0.01$ (left)  and 
$\ell=0.025$ (right), where PF control aims at stabilizing $x^*=2.5$, $\nu \approx 0.253555$. In both figures, five
runs are illustrated, $x_0=1$.}
\label{figure4}   
\end{figure}

\begin{figure}[ht]
\centering   
\includegraphics[height=.15\textheight]{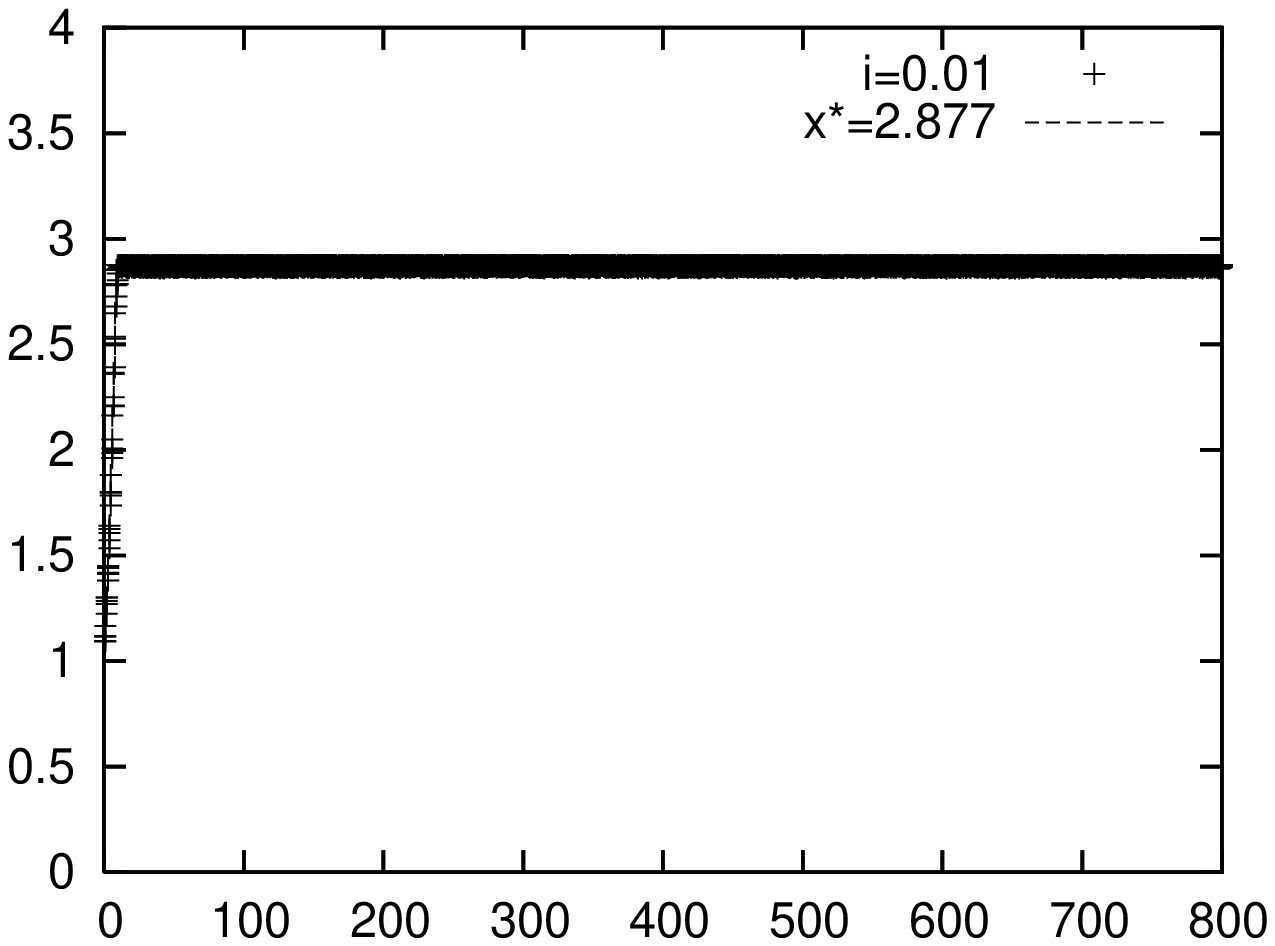}
\hspace{2mm}
\includegraphics[height=.15\textheight]{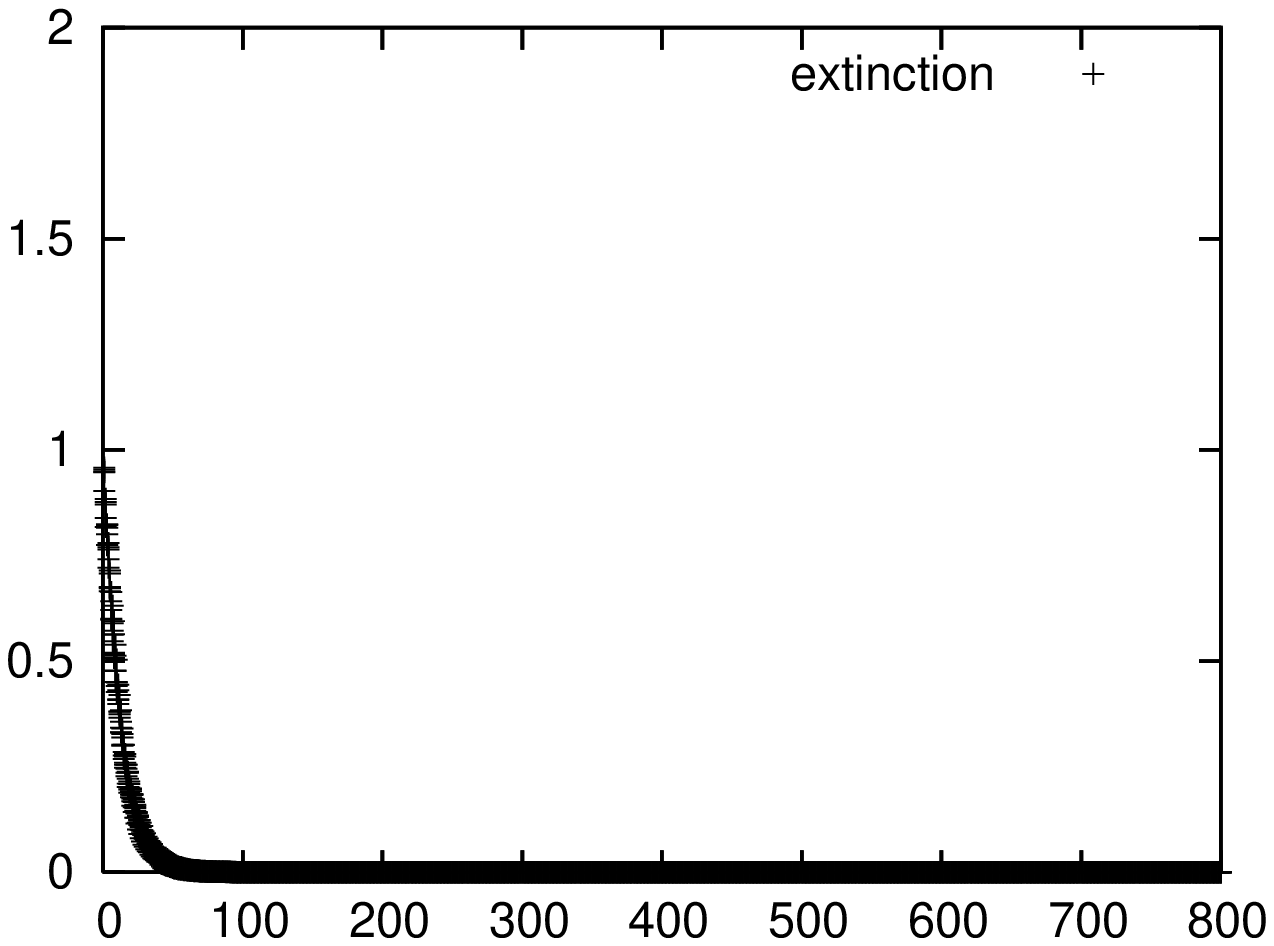}
\hspace{2mm}
\includegraphics[height=.15\textheight]{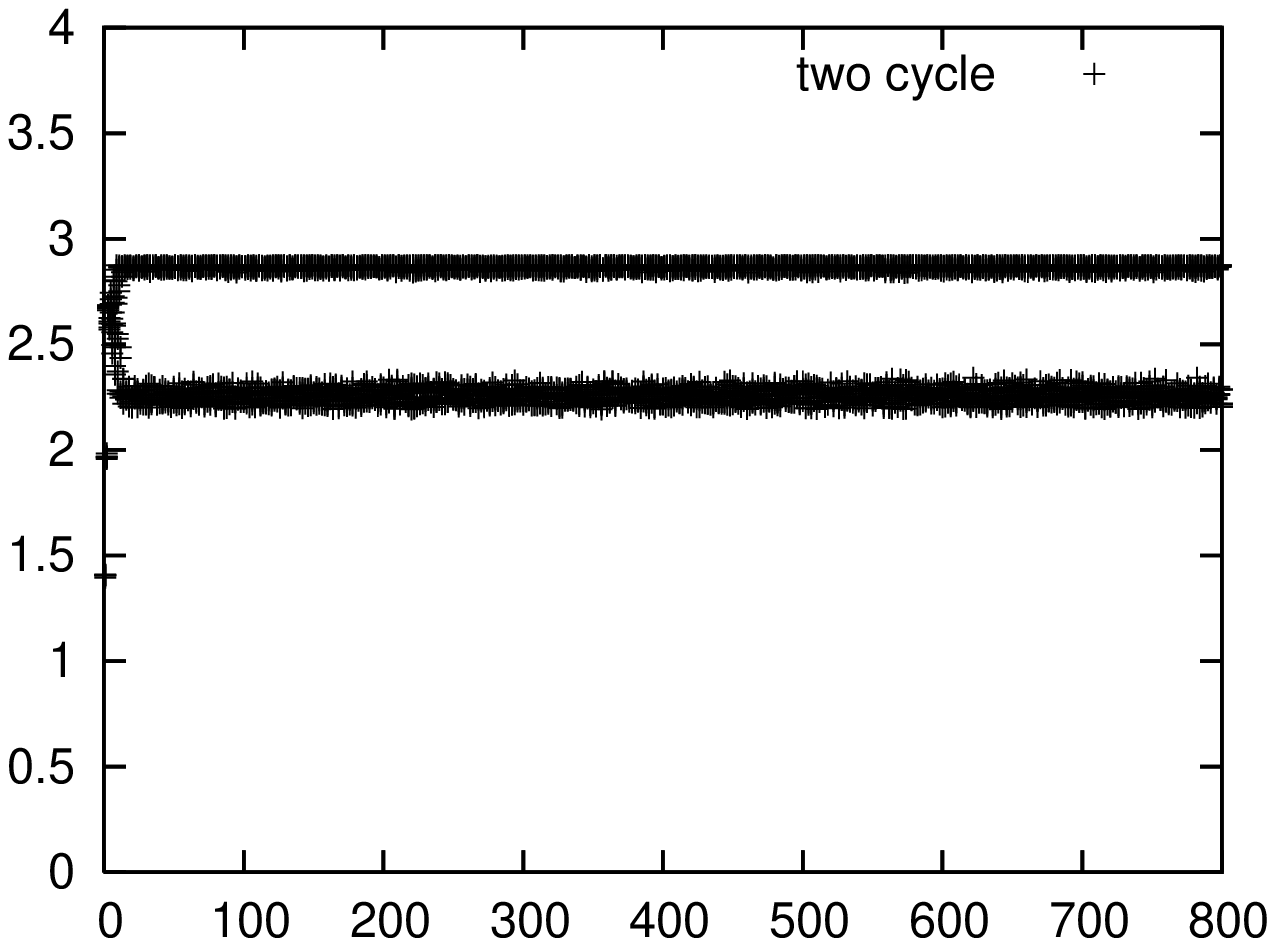}

\caption{Solutions of the difference equation with $f$ as in
(\protect{\ref{eq:CAMWA}}) and
multiplicative stochastic perturbations with $\ell=0.01$, where we stabilize
the maximum $\approx 2.877$ (left), the zero equilibrium with $\nu=0.23$ (middle) and obtained blurred cycle 
for $\nu=0.35$ (right). Everywhere we present
five runs, $x_0=1$.}
\label{figure5}
\end{figure}

In Fig.~\ref{figure6} we ullustrate stabilization of $x^*=2.5$ with $\nu \approx 0.253555$ and additive noise
with $\ell =0.01$ (left) and oscillation with $\nu=0.35$ (right).

\begin{figure}[ht]
\centering
\includegraphics[height=.18\textheight]{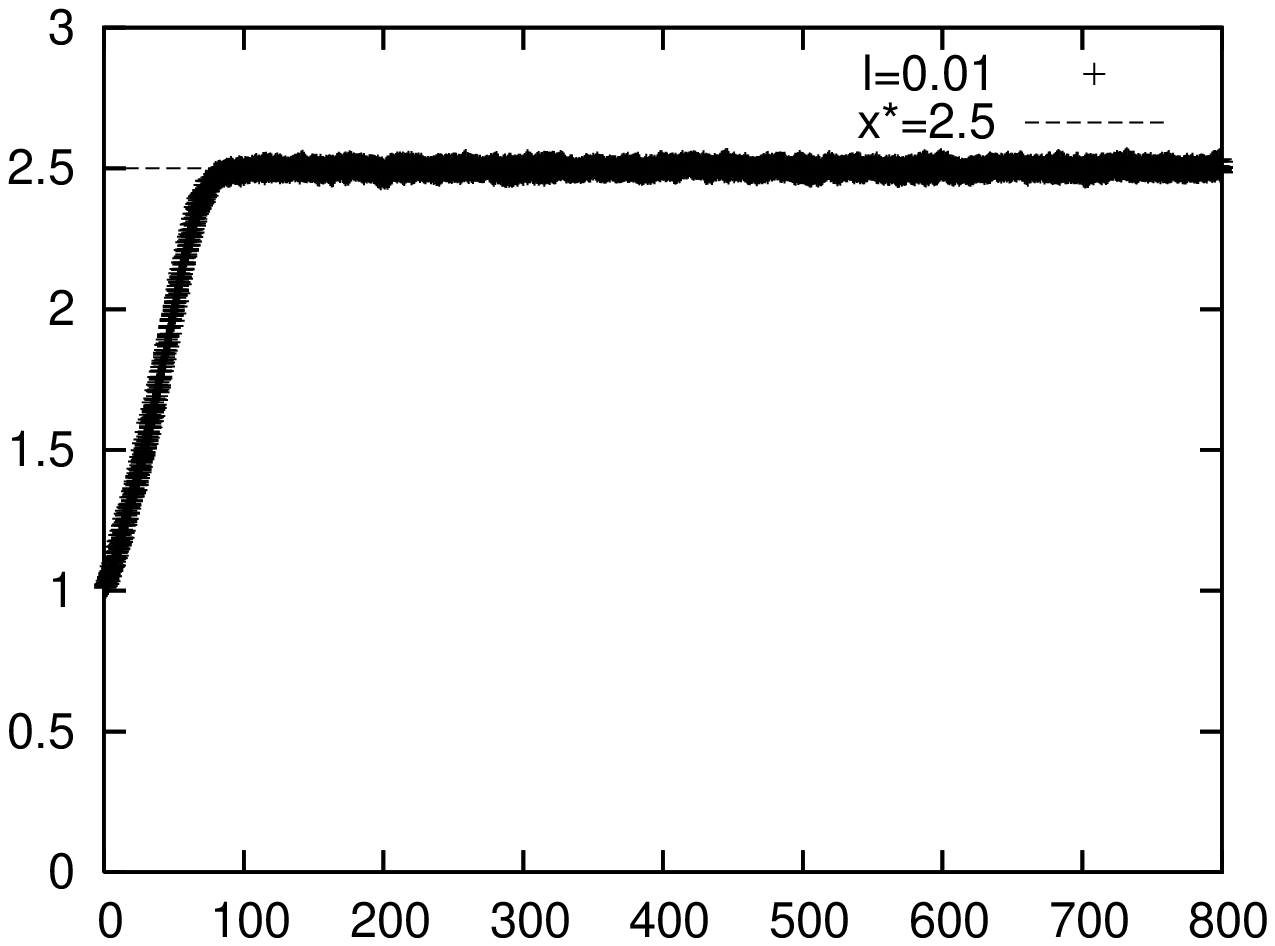}
\includegraphics[height=.18\textheight]{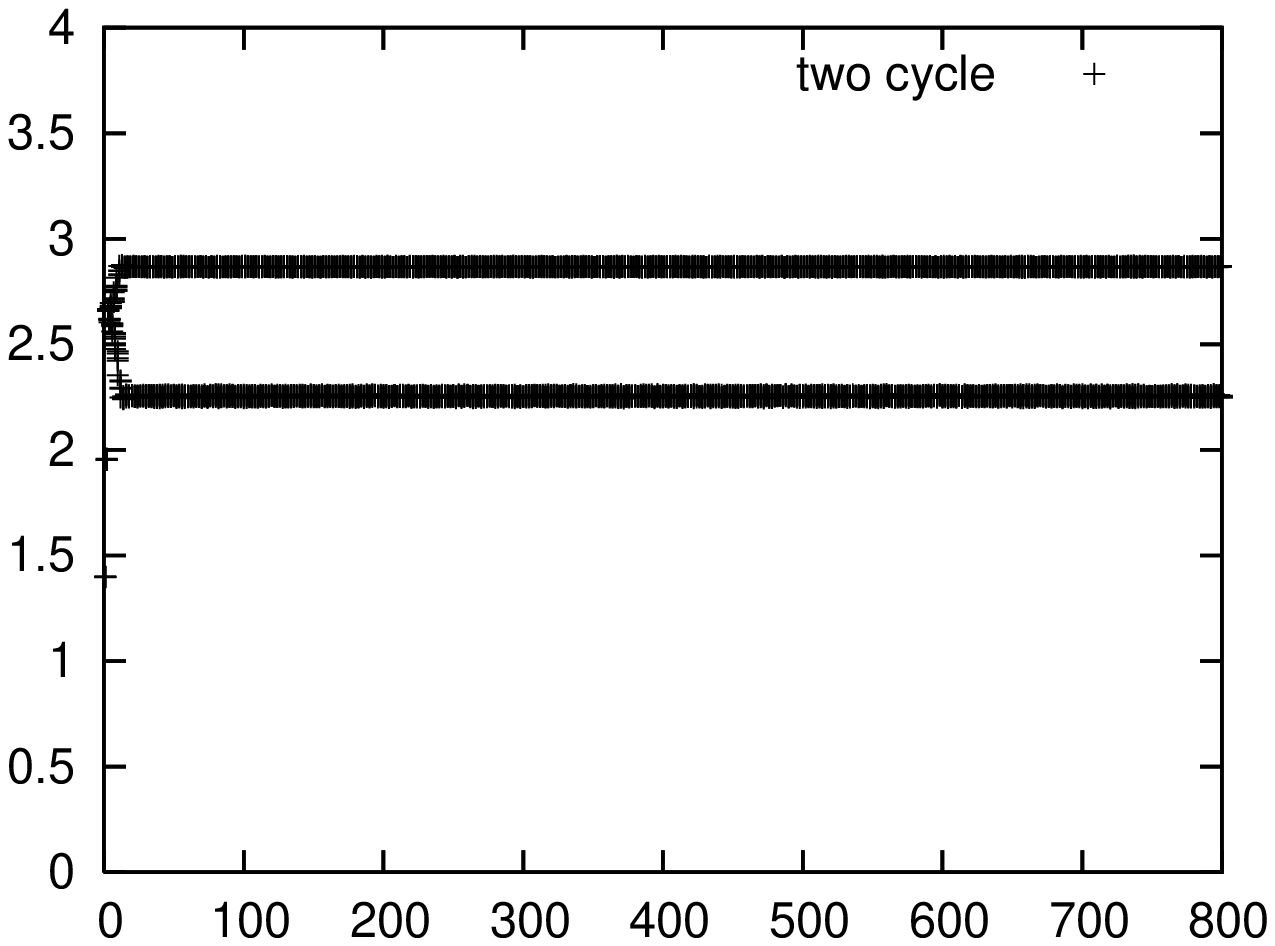}
\caption{Solutions of the difference equation with $f$ as in~(\protect{\ref{eq:CAMWA}}) and
additive stochastic perturbations with $\ell=0.01$, where $x^*=2.5$ is stabilized (left),
or there are sustainable blurred oscillations for $\nu=0.35$ (right).
In each figure, we present five runs, $x_0=1$.}
\label{figure6}   
\end{figure}
\end{example}

\section{Discussion and Open Problems}
\label{sec:discussion}

In the present paper we considered PF stabilization, and the approach is a little bit different than in
\cite{NODY, Carmona, Liz2010} where an appropriate range of stabilizing $\nu$ was established, and the stabilization 
point was found for any of such $\nu$. Here we focus on the range of points that can be stabilized with PF control,
and, once such $x^*$ is chosen, we identify the required control level. In addition, either the control or the environment can be stochastic, we get a stable blurred equilibrium as a result of the control.
Numerical examples illustrate that 
\begin{enumerate}
\item
the stochasticity in the control can lead to higher solution variation than  the stochasticity of the environment 
(compare Figs.~\ref{figure2} and \ref{figure3});
\item
there is a strong dependency of this variation on both the form of the function $f$ involved in the difference equation and the point to be stabilized 
(compare the right top and the middle figures in Fig.~\ref{figure2}).  
\end{enumerate} 

Though we deal with stochastic perturbations with a constant maximal amplitude, statements for
this amplitude tending to zero give an indication that results for a stochastic perturbation tending to zero 
with time will be similar to the non-stochastic PF control \cite{Carmona, Liz2010}. Certainly stochasticity of 
the control and the environment
can be combined.

Some relevant open problems and topics for future research are outlined below.
\begin{enumerate}
\item
Estimate the probability of certain solution bounds if the conditions of the theorems of the present paper  are not satisfied,
for example, in the case of the Allee effect.
\item
Consider unbounded, for example, normally distributed, perturbations.
\item
Explore the equation with both multiplicative and additive noise
\[
x_{n+1}=f\left( (\nu + l_1\chi_{n+1})x_n \right)+l_2\zeta_{n+1}, \,\, x_0>0, \quad n\in \mathbb N,
\]
describing a model where both the harvesting effort (or pest management efficiency) and 
the environment are stochastic.
\item
Investigate the case where not a positive equilibrium but a cycle can be stabilized,
with a control applied at each step.
\item
Consider the equation where a stochastic control is applied on certain steps only, 
so called ``impulsive" control \cite{NODY} and consider stabilization of blurred cycles. 
\end{enumerate}

\end{document}